\documentclass[11pt,]{article}
\usepackage{amsfonts,amssymb,amsmath,amscd,latexsym,makeidx,theorem, color}
\usepackage{pdfsync}
\usepackage{comment}
\title{A fractional Moser-Trudinger type inequality in one dimension and its critical points}

\author{Stefano Iula\thanks{The authors are supported by the Swiss National Science Foundation, project nr. PP00P2-144669.}\\ {\small Universit\"at Basel}\\ {\small \texttt{stefano.iula@unibas.ch} }\and Ali Maalaoui${}^*$\\ {\small A.U. Ras Al Khaimah}\\ {\small \texttt{ali.maalaoui@aurak.ae} }\and Luca Martinazzi${}^*$ \\ {\small Universit\"at Basel}\\ {\small \texttt{luca.martinazzi@unibas.ch} }}

\newtheorem{trm}{Theorem}[section]
\newtheorem{prop}[trm]{Proposition}

\newtheorem{lemma}[trm]{Lemma}

\newtheorem{rmk}{Remark}
\newtheorem{oq}{Open question}

\newcommand{\R}[1]{\mathbb{R}^{#1}}

\newcommand{\de}{\partial}
\newcommand{\ve}{\varepsilon}
\newcommand{\M}[1]{\mathcal{#1}}

\newenvironment{proof}{\noindent\emph{Proof.}}{\phantom{ } \hfill$\square$\medskip}

\DeclareMathOperator{\loc}{loc}

\begin{document}
\maketitle

\begin{abstract} We show a sharp fractional Moser-Trudinger type inequality in dimension $1$, i.e. for any interval $I\Subset\R{}$ and $p\in (1,\infty)$ there exists $\alpha_p>0$ such that
$$\sup_{u\in \tilde H^{\frac 1p,p}(I): \|(-\Delta)^\frac{1}{2p}u\|_{L^p(I)}\le 1} \int_I e^{\alpha_p |u|^\frac{p}{p-1}}dx = C_p|I|,$$
and $\alpha_p$ is optimal in the sense that
$$\sup_{u\in \tilde H^{\frac 1p,p}(I): \|(-\Delta)^\frac{1}{2p}u\|_{L^p(I)}\le 1} \int_I h(u) e^{\alpha_p |u|^\frac{p}{p-1}}dx = +\infty,$$
for any function $h:[0,\infty) \to [0,\infty)$ with $\lim_{t\to\infty}h(t)=\infty$.
Here $\tilde H^{\frac 1p,p}(I)=\{u\in L^p(\R{}): (-\Delta)^\frac{1}{2p}u\in L^p(\R{}), \mathrm{supp}(u)\subset \bar I\}$.

Restricting ourselves to the case $p=2$ we further consider for $\lambda>0$ the functional
$$J(u):=\frac{1}{2}\int_{\R{}}|(-\Delta)^\frac14 u|^2 dx- \lambda\int_I \left(e^{\frac12 u^2}-1\right)dx,\quad u\in \tilde H^{\frac{1}{2},2}(I),$$
and prove that it satisfies the Palais-Smale condition at any level $c\in (-\infty,\pi)$.
We use these results to show that the equation
$$(-\Delta)^\frac12 u =\lambda u e^{\frac{1}{2}u^2}\quad \text{in }I$$
has a positive solution in $\tilde H^{\frac12,2}(I)$ if and only if $\lambda\in (0,\lambda_1(I))$, where $\lambda_1(I)$ is the first eigenvalue of $(-\Delta)^\frac12$ on $I$. This extends to the fractional case some previous results proven by Adimurthi for the Laplacian and the $p$-Laplacian operators.

Finally with a technique by Ruf we show a fractional Moser-Trudinger inequality on $\R{}$.

\end{abstract}

\section{Introduction}
According to a celebrated result in analysis, if $\Omega\subset\R{n}$ is an open set with finite measure $|\Omega|$, $k$ is a positive integer smaller than $n$, and $p\in [1,\frac{n}{k})$, then the Sobolev space $W^{k,p}_0(\Omega)$ embeds continuously into $L^{\frac{np}{n-kp}}(\Omega)$. It is also known that in the borderline case $p=\frac{n}{k}$ one has $W^{k,\frac{n}{k}}_0(\Omega)\not\subset L^\infty(\Omega)$. On the other hand, as shown by Pohozaev \cite{poh}, Trudinger \cite{tru}, Yudovich \cite{yud} and others, for the case $k=1$ one has
$$W^{1,n}_0(\Omega)\subset \left\{u\in L^1(\Omega):E_\beta(u):=\int_\Omega e^{\beta |u|^{\frac{n}{n-1}}}dx<\infty  \right\},\quad \text{for any }\beta <\infty,$$
and the functional $E_\beta$ is continuous on $W^{1,n}_0(\Omega)$.
This embedding was complemented with a sharp inequality by Moser \cite{mos}, the so-called Moser-Trudinger inequality: 
\begin{equation}\label{stimaMT0}
\sup_{u\in W^{1,n}_0(\Omega),\;\|\nabla u\|_{L^n(\Omega)}\le 1} \int_\Omega e^{\alpha_n |u|^\frac{n}{n-1}}dx \le C|\Omega|,\quad \alpha_n:=n\omega_{n-1}^\frac{1}{n-1},
\end{equation}
where $\omega_{n-1}$ is the volume of the unit sphere in $\R{n}$, and the constant $\alpha_n$ is sharp.

Several generalizations and applications of the Moser-Trudinger inequality have appeared in the course of the last decades. In this work we investigate a fractional version of \eqref{stimaMT0}, and in particular its sharpness, restricting our attention to dimension $1$. Moreover we will consider a functional associated to it and discuss its critical points.

\medskip

In order to state the main results of the paper we first introduce some  relevant function spaces.
Let $s\in (0,1)$. We consider the space of functions $L_{s}(\mathbb{R})$ defined by
\begin{equation}
L_{s}(\mathbb{R})=\left\{ u\in L^{1}_{\loc}(\mathbb{R}): \int_{\mathbb{R}}\frac{|u(x)|}{1+|x|^{1+2s}}dx <\infty \right\}.
\end{equation}
For a function $u\in L_{s}(\mathbb{R})$  we define $(-\Delta)^{s}u$ as a tempered distribution as follows:
\begin{equation}\label{deffrlap}
\langle (-\Delta)^{s}u,\varphi \rangle := \int_{\R{}} u(-\Delta)^{s}\varphi dx, \quad \varphi\in \M{S},
\end{equation}
where $\M{S}$ denotes the Schwartz space of rapidly decreasing smooth functions and for $\varphi\in\M{S}$ we set
$$(-\Delta)^s\varphi:=\mathcal{F}^{-1}(|\cdot |^{2s}\hat{\varphi}).$$ Here the Fourier transform is defined by
$$ \hat \varphi(\xi)\equiv \mathcal{F} \varphi(\xi):=\frac{1}{\sqrt{2\pi}}\int_{\R{}}e^{-ix\xi}\varphi(x)\, dx.$$
Notice that the convergence of the integral in \eqref{deffrlap} follows from the fact that for $\varphi\in \M{S}$ one has
$$|(-\Delta)^s\varphi(x)| \le C(1+|x|^{1+2s})^{-1}.$$

\medskip

For $s\in (0,1)$ and $p\in [1,\infty]$ we define the Bessel-potential space
\begin{equation}\label{BPS}
H^{s,p}(\R{}):=\left\{u\in L^p(\R{}) \colon (-\Delta)^{\frac s2}u \in L^p(\R{})\right\},
\end{equation}
and its subspace 
\begin{equation}\label{BPSS}
\tilde H^{s,p}(I):=\{u\in L^p(\R{}): u\equiv 0\text{ in }\R{}\setminus I,\;(-\Delta)^{\frac s2} u\in L^p(\R{})\},
\end{equation}
where $I\Subset\R{}$ is a bounded interval. Both spaces are endowed with the norm

\begin{equation}\label{Hspnorm}
\|u \|^p_{H^{s,p}(\R{})}:=\|u\|^p_{L^p(\R{})}+ \|(-\Delta)^{\frac s2}u\|^p_{L^p(\R{})}.
\end{equation}

\subsection{A fractional Moser-Trudinger type inequality}

The first result that we shall prove is a fractional Moser-Trudinger type inequality:

\begin{trm}\label{MT2}
For any $p\in(1,+\infty)$ set $p'=\tfrac{p}{p-1}$ and
\begin{equation}\label{alphap}
\alpha_p:=\frac{1}{2}\left[2\cos\left(\frac{\pi}{2p}\right)\Gamma\left(\frac{1}{p}\right)\right]^{p'},\quad \Gamma(z):=\int_0^{+\infty} t^{z-1}e^{-t}\, dt.
\end{equation}
Then for any interval $I \Subset \R{}$ and $\alpha\leq \alpha_p$ we have

\begin{equation}\label{stimaMT}
\sup_{u\in \tilde H^{\frac1p,p}(I), \;\|(-\Delta)^{\frac{1}{2p}}u\|_{L^{p}(I)}\leq 1}\int_{I}\left(e^{\alpha |u|^{p'}}-1\right)dx = C_p|I|,
\end{equation}
and $\alpha=\alpha_p$ is the largest constant for which \eqref{stimaMT} holds. In fact for any function $h:[0,\infty) \to [0,\infty)$ with 
\begin{equation}\label{condh}
\lim_{t\to\infty} h(t)=\infty
\end{equation}
we have
\begin{equation}\label{stimaMT2s}
\sup_{u\in \tilde H^{\frac1p,p}(I),\; \|(-\Delta)^\frac{1}{2p} u\|_{L^p(I)}\leq 1} \int_{I}h(u) \left(e^{\alpha_p |u|^{p'}} -1\right) dx =\infty.
\end{equation}

\end{trm}

\begin{rmk}\label{rm1}
Notice that in \eqref{stimaMT}, instead of the standard $H^{{\frac 1p},p}$-norm defined in \eqref{Hspnorm}, we are using the smaller  norm $\|u\|^{*}:= \|(-\Delta)^{\frac{1}{2p}}u\|_{L^p(I)}$. This turns out to be equivalent to the full norm $\|u\|_{H^{\frac 1p,p}(\R{})}$ on $\tilde H^{\frac 1p,p}(I)$. This fact does not appear to be obvious, but one can prove it as follows. By Theorem $7.1$ in \cite{Gr0} the operator $T:u\mapsto ((-\Delta)^\frac{1}{2p} u)|_I$ is Fredholm from $\tilde H^{\frac 1p,p}(I)$ ($=H^{\frac{1}{2p}(\frac 1p)}_p(\bar I)$ in the notation of \cite{Gr0}) into $L^p(I)$. Moreover $T$ is injective by Lemma \ref{lemmagreen3} below. This implies that
$$\|u\|_{H^{\frac 1p,p}(\R{})}\le C \|Tu \|_{L^p(I)}= C \|u\|^{*}, \quad \text{for every }u\in \tilde H^{\frac 1p,p}(I).$$
\end{rmk}

Recently A. Iannizzotto and M. Squassina \cite[Cor. 2.4]{IS} proved a subcritical version of \eqref{stimaMT} in Theorem \ref{MT2} in the case $p=2$, namely
$$\sup_{u\in \tilde H^{\frac12,2}(I)\,:\, \|(-\Delta)^\frac14 u\|_{L^2(\R{})}\le 1}\int_I e^{\alpha u^2}dx \le C_\alpha|I|,\quad \text{for }\alpha<\pi.$$

\subsection{Palais-Smale condition and critical points}

In the rest of this paper we will focus on the case $p=2$, and denote 
\begin{equation}\label{defH}
H:=\tilde{H}^{\frac{1}{2},2}(I),\quad \|u\|_{H}:= \|(-\Delta)^{\frac 14}u\|_{L^2(\R{})}. 
\end{equation}
By Remark \ref{rm1} also this norm is equivalent to the full $H^{\frac12,2}$-norm on $\tilde{H}^{\frac{1}{2},2}(I)$. This also follows from the following Poincar\'e-type inequality (see e.g. \cite[Lemma 6]{SV}):
\begin{equation}\label{poinc}
\|u\|_{L^2(I)}^2\le \frac{1}{\lambda_1(I)} \|(-\Delta)^\frac14 u\|_{L^2(\R{})}^2\quad \text{for }u\in   \tilde{H}^{\frac{1}{2},2}(I),
\end{equation}
where $\lambda_1>0$ is the first eigenvalue of $(-\Delta)^\frac12$ on $\tilde{H}^{\frac{1}{2},2}(I)$, see Lemma \ref{phi1} in the appendix.

We now investigate the existence of critical points of functionals associated to inequality \eqref{stimaMT} in the case $p=2$. Since we often integrate by parts and $(-\Delta)^s u$ is not in general supported in $I$ even if $u\in C^\infty_c(I)$, it is more natural to consider the slightly weaker inequality 
\begin{equation}\label{stimaMT2}
\sup_{u\in H, \;\|u\|_H^2 \leq 2\pi}\int_{I}\left(e^{\frac12 u^2}-1\right)dx = C|I|,
\end{equation}
where we use the slightly different norm given in \eqref{defH}. The reason for using the constant $\frac12$ instead of $\alpha_2=\pi$ in the exponential and having $\|u\|^2_H\le 2\pi$ instead of $\|u\|_H^2\le 1$ is mostly cosmetic, and becomes more apparent when studying the blow-up behaviour of critical points of functionals associated to \eqref{stimaMT2} (see \eqref{defJ} below,  and compare to \cite{MMS} and \cite{mar1}).

\medskip

We want to investigate the existence of solutions of the non-local equation
\begin{equation}\label{eq0}
(-\Delta)^\frac12 u=\lambda ue^{\frac{1}{2}u^2}\quad \text{in }I, \quad u\equiv 0 \text{ in }\R{}\setminus I,
\end{equation}
which is the equation satisfied by critical points of the functional $E:M_\Lambda\to\R{}$, where
$$E(u)=\int_{I}\left(e^{\frac12 u^2}-1\right)dx, \quad M_\Lambda:=\{u\in H:\|u\|_H^2=\Lambda\},$$
$\Lambda>0$ is given, $\lambda$ is a Lagrange multiplier and $E$ is well defined on $M_\Lambda$ thanks to Lemma \ref{Lp} below. Since with this variational interpretation of \eqref{eq0} it is not possible to prescribe $\lambda$, we will follow the approach of Adimurthi and see solutions of \eqref{eq0} as critical points of the functional
\begin{equation}\label{defJ}
J:H\to \R{},\quad J(u)=\frac{1}{2} \|u\|_{H}^2 -\lambda \int_{I}\left(e^{\frac{1}{2}u^{2}}-1\right) dx.
\end{equation}
Again $J$ is well-defined on $H$ by Lemma \ref{Lp}. Moreover it is differentiable by Lemma \ref{lemmaC1} below, and its derivative is given by
$$\langle J'(u),v\rangle:=\frac{d}{dt}J(u+tv)\bigg|_{t=0} = (u,v)_H -\lambda\int_I uv e^{\frac{1}{2}u^2} dx,  $$
for any $u,v\in H$, where
$$(u,v)_H:= \int_{\R{}} (-\Delta)^{\frac 14}u(-\Delta)^{\frac 14}v\, dx.$$
In particular we have that if $u\in H$ and $J'(u)=0$, then $u$ is a weak solution of Problem \eqref{eq0} in the sense that
\begin{equation}\label{eq1}
(u,v)_H=\lambda\int_I uv e^{\frac{1}{2}u^2}dx, \quad \text{for all }v\in H.
\end{equation}
That this Hilbert-space definition of \eqref{eq0} is equivalent to the definition in sense of tempered distributions given by \eqref{deffrlap} is discussed in the introduction of \cite{MMS}.

To find critical points of $J$ we will follow a method of Nehari, as done by Adimurthi \cite{A}. An important point will be to understand whether $J$ satisfies the Palais-Smale condition or not.
We will prove the following:

\begin{prop}\label{trmex}
The functional $J$ satisfies the Palais-Smale condition at any level $c \in (-\infty, \pi )$, i.e.  any sequence $(u_k)$ with
\begin{equation}\label{condPS}
J(u_k)\to c\in (-\infty,\pi),\quad \|J'(u_k)\|_{H'}\to 0\quad\text{as }k\to\infty
\end{equation}
admits a subsequence strongly converging in $H$.
\end{prop}

\begin{trm}\label{ex}
Let $I\subset \R{}$ be a bounded interval and $\lambda_1(I)$ denote the first eigenvalue of $(-\Delta)^\frac12$ on $H=\tilde H^{\frac12,2}(I)$. Then for every $\lambda\in (0,\lambda_{1}(I))$ Problem \eqref{eq0} has at least one \emph{positive} solution $u\in H$ in the sense of \eqref{eq1}. 
When $\lambda\ge \lambda_1(I)$ or $\lambda\le 0$ Problem \eqref{eq0} has no non-trivial non-negative solutions.
\end{trm}

To prove Theorem \ref{ex} one constructs a sequence $(u_k)$ which is almost of Palais-Smale type for $J$, in the sense that $J(u_k)\to \bar c$ for some $\bar c\in \R{}$ and $\langle J'(u_k), u_k\rangle =0$. Then a modified version of Proposition \ref{trmex} is used, namely Lemma \ref{comp} below. In order to do so, it is crucial to show that $\bar c<\pi$ (Lemma \ref{lemmaaJ2} below) and this will follow from \eqref{stimaMT2s} with $p=2$ and $h(t)=|t|^2$. Interestingly, in the general case $s>1$, $n\ge 2$, $p=\frac{n}{s}$, the analog of \eqref{stimaMT2s} is known only when $s$ is integer or when $h$ satisfies $\lim_{t\to\infty}(t^{-p'}h(t))=\infty$ (see \cite{mar2} and Remark \ref{rmk1} below). 

Both Proposition \ref{trmex} and Theorem \ref{ex} were first proven by Adimurthi \cite{A} in dimension $n\ge 2$ with $(-\Delta)^\frac{1}{2}$ replaced by the $n$-Laplacian. 

\medskip

Let us briefly discuss the blow-up behaviour of solutions to \eqref{eq0}. Extending previous works in even dimension (see e.g. \cite{AS}, \cite{dru}, \cite{mar1}, \cite{RS}) the second and third authors and Armin Schikorra \cite{MMS} studied the blow-up of sequences of solutions to the equation
$$(-\Delta)^\frac{n}{2} u=\lambda u e^{\frac{n}{2}u^2}\quad \text{in }\Omega\Subset\R{n}$$
with suitable Dirichlet-type boundary conditions when $n$ is odd. The moving plane technique for the fractional Laplacian (see \cite{BLW}) implies that a non-negative solution to \eqref{eq0} is symmetric and monotone decreasing from the center of $I$. Then it is not difficult to check that in dimension $1$ Theorem 1.5 and Proposition 2.8 of \cite{MMS} yield:

\begin{trm}\label{compact}
Fix $I=(-R,R)\Subset\R{}$ and let $(u_k)\subset H=\tilde H^{\frac12,2}(I)$ be a sequence of non-negative solutions to
\begin{equation}\label{eq0k}
(-\Delta)^\frac12 u_k=\lambda_k u_ke^{\frac{1}{2}u_k^2}\quad \text{in }I,
\end{equation}
in the sense of \eqref{eq1}. Let $m_k:=\sup_I u_k$ and assume that
$$\Lambda:=\limsup_{k\to\infty}\|u_k\|_H^2<\infty.$$
Then up to extracting a subsequence we have that either
\begin{itemize}
\item[(i)] $u_k\to u_\infty$ in $C^\ell_{\loc}(I)\cap C^0(\bar I)$ for every $\ell\ge 0$, where $u_\infty\in C^\ell_{\loc}(I)\cap C^0(\bar I)\cap H$ solves
\begin{equation}\label{eqinfty}
(-\Delta)^\frac12 u_\infty=\lambda_\infty u_\infty e^{\frac{1}{2}u_\infty^2}\quad \text{in }I,
\end{equation}
for some $\lambda_\infty\in (0,\lambda_1(I))$, or
\item[(ii)] $u_k\to u_\infty$ weakly in $H$ and strongly in $C^0_{\loc}(\bar I\setminus \{0\})$ where $u_\infty$ is a solution to \eqref{eqinfty}. Moreover, setting $r_k$ such that $\lambda_k r_k m_k^2e^{\frac12 m_k^2}$ and
\begin{equation}\label{etak}
\eta_k(x):=m_k(u_k(r_kx)-m_k)+\log 2,\quad \eta_\infty(x):=\log\left(\frac{2}{1+|x|^2}\right),
\end{equation}
one has $\eta_k\to \eta_\infty$ in $C^\ell_{\loc}(\R{})$ for every $\ell\ge 0$ and $\Lambda\ge \|u_\infty\|_H^2+2\pi$.
\end{itemize}
\end{trm}

The function $\eta_\infty$ appearing in \eqref{etak} solves the equation
$$(-\Delta)^\frac12 \eta_\infty = e^{\eta_\infty}\quad \text{in }\R{},$$
which has been recently interpreted in terms of holomorphic immersions of a disk (or the half-plane) by Francesca Da Lio, Tristan Rivi\`ere and the third author \cite{DLMR}.

Theorem \ref{compact} should be compared with the two dimensional case, where the analogous equation $-\Delta u =\lambda ue^{u^2}$ on the unit disk
has a more precise blow-up behaviour, see e.g. \cite{AP}, \cite{AS}, \cite{dru}, \cite{MM}.

\subsection{A fractional Moser-Trudinger type inequality on the whole $\R{}$}

When replacing a bounded interval $I$ by $\R{}$, an estimate of the form \eqref{stimaMT} cannot hold, for instance because of the scaling  of \eqref{stimaMT}, or simply because the quantity $\|(-\Delta)^\frac{1}{2p}u\|_{L^p(\R{})}$ vanishes on constants. This suggests to use the full Sobolev norm including the term $\|u\|_{L^p(I)}$ (see Remark \ref{rm1}). This was done by Bernhard Ruf \cite{Ruf} in the case of $H^{1,2}(\R{2})$. We shall adapt his technique to the case $H^{\frac{1}{2},2}(\R{})$.

\begin{trm}\label{MT3}
We have 
\begin{equation}\label{stimaMT3}
\sup_{u\in H^{\frac12,2}(\mathbb{R}), \;\|u\|_{{H}^{\frac12,2}(\R{})}\leq 1} \int_{\mathbb{R}}\left( e^{\pi u^{2}}-1 \right) dx <\infty,
\end{equation}
where $\|u\|_{H^{\frac12,2}(\R{})}$ is defined in \eqref{Hspnorm}. Moreover, for any function
$h:[0,\infty)\to [0,\infty)$ satisfying
\begin{equation}\label{condh2}
\lim_{t\to\infty} (t^{-2}h(t))=\infty
\end{equation}
we have
\begin{equation}\label{stimaMT3s}
\sup_{u\in H^{\frac12,2}(\mathbb{R}),\;  \|u\|_{{H}^{\frac12,2}(\R{})}\leq 1} \int_{\mathbb{R}}h(u)\left( e^{\pi u^{2}}-1\right) dx =\infty.
\end{equation}
In particular the constant $\pi$ in \eqref{stimaMT3} is sharp.
\end{trm}

A main ingredient in the proof of \eqref{stimaMT3} is a fractional P\'olya-Szeg\H{o} inequality which seems to be known only in the $L^2$ setting, being based mainly on Fourier transform techniques.

\begin{oq} Does an $L^p$-version of Theorem \ref{MT3} hold, i.e. can we replace $H^{\frac12,2}$ with $H^{\frac1p,p}$ in \eqref{stimaMT3}? 
\end{oq}

The reason for requiring \eqref{condh2} in Theorem \ref{MT3} (contrary to Theorem \ref{MT2}, where \eqref{condh} suffices) is that the test functions for \eqref{stimaMT3s} will be constructed using a cut-off procedure, and due to the non-local nature of the $H^{\frac12,2}$-norm, giving a precise estimate for the norm of such test functions is difficult.

\begin{oq} In analogy with Theorem \ref{MT2}, does \eqref{stimaMT3s} hold for every $h$ satisfying \eqref{condh}?
\end{oq}

In the following sections we shall prove Theorems \ref{MT2}, \ref{ex} and \ref{MT3}, and Proposition \ref{trmex}. In the appendix we collected some useful results about fractional Sobolev spaces and fractional Laplace operators.

\section{Theorem \ref{MT2}}\label{MTI}
\subsection{Idea of the proof}

The following analog of \eqref{stimaMT}
\begin{equation}\label{stimaMT2bis}
\sup_{u=c_pI_\frac1p* f\,:\, \mathrm{supp}(f)\subset\bar I, \,\|f\|_{L^{p}(I)}\leq 1}\int_{I}e^{\alpha_p |u|^{p'}}dx = C_p|I|, \quad I_{\frac1p}(x):=|x|^{\frac1p -1}
\end{equation}
is well-known (also in higher dimension, see e.g. \cite[Theorem $3.1$] {XZ}), since it follows easily from the Theorem 2 in \cite{ada}, up to choosing $c_p$ so that
\begin{equation}\label{defIp}
c_p(-\Delta)^\frac1{2p} I_\frac1p =\delta_0,
\end{equation}
compare to Lemma \ref{green2} below.

In \eqref{stimaMT2bis} one requires that the support of $f=(-\Delta)^\frac{1}{2p}u$ is bounded; following Adams \cite{ada} one would be tempted to write $u=I_{\frac{1}{p}}*(-\Delta)^\frac{1}{2p}u$ and apply \eqref{stimaMT2bis}, but the support of $(-\Delta)^\frac{1}{2p} u$ is in general not bounded, when $u$ is compactly supported.

In order to circumvent this issue, we rely on a Green representation formula of the form
$$u(x)=\int_I G_\frac{1}{2p}(x,y)(-\Delta)^{\frac{1}{2p}}u(y)dy,$$
and show that 
$|G_{\frac{1}{2p}}(x,y)|\le I_\frac{1}{p}(x-y)$ for $x\ne y$. This might follow from the explicit formula of $G_{s}(x,y)$, which is known on an interval, see e.g. \cite{BGR} and \cite{CB}, but we prefer to follow a more self-contained path, only using the maximum principle.

More delicate is the proof of \eqref{stimaMT2s}.
We will construct functions $u$ supported in $\bar I$ with $(-\Delta)^\frac{1}{2p} u=f$ for some prescribed function $f\in L^p(I)$ suitably concentrated. Then with a barrier argument we will show that $u\in \tilde H^{\frac1p,p}(I)$, i.e. $(-\Delta)^\frac{1}{2p} u\in L^p(\R{})$.  This is not obvious because $(-\Delta)^\frac{1}{2p}$ is a non-local operator and even if $u\equiv 0$ in $I^c$, $(-\Delta)^\frac{1}{2p}u$ does not vanish outside $I$, and a priori it could even concentrate on $\de I$. 

\begin{rmk}\label{rmk1} An alternative approach to \eqref{stimaMT2s} uses the Riesz potential and a cut-off function $\psi$, as done in \cite{mar2} following a suggestion of A. Schikorra. This works in every dimension and for arbitrary powers of $-\Delta$, but it is less efficient in the sense that the $\|(-\Delta)^ s \psi\|_{L^p}$ is not sufficiently small, and \eqref{stimaMT2s} (or its higher-order analog) can be proven only for functions $h$ such that $\lim_{t\to\infty} (t^{-p'}h(t))=\infty$. On the other hand, the approach used here to prove \eqref{stimaMT2s} for every $h$ satisfying \eqref{condh} does not work for higher-order operators, since for instance if for $\Omega\Subset \R{4}$ we take $u\in W^{1,2}_0(\Omega)$ solving $\Delta u = f\in L^2(\Omega)$, then we do not have in general $u\in W^{2,2}(\R{4})$. 
\end{rmk}

\subsection{Proof of Theorem \ref{MT2}}
By a simple scaling argument it suffices to prove \eqref{stimaMT} for a given interval, say $I=(-1,1)$.

\begin{lemma}\label{green2}
For $s\in \left(0,\frac12\right)$ the fundamental solution of $(-\Delta)^s$ on $\R{}$ is
$$F_s(x)=\frac{1}{2\cos(s\pi)\Gamma(2s) |x|^{1-2s} },$$
i.e. $(-\Delta)^s F_s=\delta_0$ in the sense of tempered distributions.
\end{lemma}

\begin{proof} This follows easily e.g. from Theorem 5.9 in \cite{LL}.
\end{proof}

\begin{lemma}\label{lemmagreen3} 
Fix $s\in \left(0,\frac12\right)$. For any $x\in I=(-1,1)$ let $g_x\in C^\infty(\R{})$ be any function with $g_x(y)=F_s(x-y)$ for $y\in I^c$. Then there exists $H_s(x,\cdot )\in \tilde H^{s,2}(I)+g_x$ unique solution to
\begin{equation}\label{H}
\left\{
\begin{array}{ll}
(-\Delta)^sH_s(x,\cdot)=0&\text{in } I\\
H_s(x,\cdot )=g_x&\text{in } \R{}\setminus I
\end{array}
\right.
\end{equation}
and the function
$$G_s(x,y):= F_s(x-y)-H_s(x,y), \qquad (x,y)\in I\times\R{}$$
is the Green function of $(-\Delta)^s$ on $I$, i.e. for $x\in I$ it satisfies 
\begin{equation}\label{eqgreen}
\left\{
\begin{array}{ll}
(-\Delta)^sG_s(x,\cdot)=\delta_x&\text{in } I\\
G(x,y)=0&\text{for }y\in \R{}\setminus I.
\end{array}
\right.
\end{equation}
Moreover
\begin{equation}\label{G<F}
0<G_s(x,y)\leq F_s(x-y) \quad \text{for }y\ne x\in I.
\end{equation}
Finally, for any function $u\in \tilde H^{2s,p}(I)$ ($p\in [1,\infty)$) we have
\begin{equation}\label{eqgreenb}
u(x) = \int_I G_s(x,y)(-\Delta)^s u(y)dy,\quad  \text{for a.e. }x\in I,
\end{equation}
where the right-hand side is well defined for a.e. $x\in I$ thanks to \eqref{G<F} and Fubini's theorem.
\end{lemma}

\begin{rmk}
The first equations in \eqref{H} above and in \eqref{eqgreen} below are intended in the sense of distribution, compare to \eqref{deffrlap}.
\end{rmk}

\begin{proof} 
The existence and non-negativity of $H_s(x,\cdot)$ for every $x\in I$ follow from Theorem \ref{trmexist} and Proposition \ref{maxprinc2} in the Appendix.
The next claim, namely \eqref{eqgreen}, follows at once from Lemma \ref{green2} and \eqref{H}.

We show now that $G(x,y)\geq 0$ for every $(x,y)\in I\times I$. We claim that
\begin{equation}\label{Hcont}
\lim_{y\to \pm 1}H_s(x,y)= H_s(x,\pm 1)=F_s(x \mp1),
\end{equation}
hence $G_s(x,y )\to 0$ as $y\to \de I$, and by Silvestre's maximum principle, Proposition \ref{maxprinc} below, we also have $G_s(x,\cdot )\ge 0$ for every $x\in I$, hence also \eqref{G<F} follows.
For the proof of \eqref{Hcont} notice that
$$\tilde H_s(x,\cdot):= H_s(x,\cdot)-g_x\in \tilde H^{s,2}(I)$$
satifies
$$
\left\{
\begin{array}{ll}
(-\Delta)^s\tilde H_s(x,\cdot)=- (-\Delta)^s g_x&\text{in } I\\
\tilde H_s(x,\cdot )=0 &\text{in } \R{}\setminus I
\end{array}
\right.
$$
and $((-\Delta)^s g_x)|_I\in L^\infty( I)$ by Proposition \ref{lapint} (we are using that $g_x\in C^\infty(\R{})$), hence Proposition \ref{trmbordo} gives $\tilde H_s(x,y)\to 0$ as $y\to\de I$, and \eqref{Hcont} follows at once.

To prove \eqref{eqgreenb}, let us start considering $u\in C^\infty_c(I).$ Then, according to \eqref{eqgreen}, we have
$$u(x)=\langle \delta_x,u\rangle=\langle (-\Delta)^s G_s(x,\cdot),u\rangle=\int_I G_s(x,y)(-\Delta)^su(y)dy.$$
Given now $u\in \tilde H^{2s,p}(I)$, let $(u_k)_{k\in \mathbb{N}}\subset C^\infty_c(I)$ converge to $u$ in $\tilde H^{2s,p}(I)$, i.e.
$$u_k\to u,\quad (-\Delta)^s u_k\to (-\Delta)^s u\quad \text{in }L^p(\R{}), \text{ hence in }L^1(I),$$
see Lemma \ref{lemmadens}. Then
$$u \overset{L^1(I)}\longleftarrow u_k=\int_I G_s(\cdot ,y)(-\Delta)^su_k(y)dy\overset{L^1(I)}\longrightarrow \int_I G_s(\cdot ,y)(-\Delta)^su(y)dy,$$
the convergence on the right following from \eqref{G<F} and Fubini's theorem:
\[
\begin{split}
\int_I &\left|\int_I G_s(x,y)\left[(-\Delta)^su_k(y)- (-\Delta)^su(y)\right]dy \right|dx\\
&\le \int_I \int_I F_s(x-y)\left|(-\Delta)^su_k(y)- (-\Delta)^su(y)\right| dx dy\\
&\le \sup_{y\in I}\|F_s\|_{L^1(I-y)} \|(-\Delta)^su_k- (-\Delta)^su\|_{L^1(I)}\to 0
\end{split}
\]
as $k\to \infty$. Since the convergence in $L^1$ implies the a.e. convergence (up to a subsequence), \eqref{eqgreenb} follows.
\end{proof}

\noindent\emph{Proof of Theorem \ref{MT2}.} 
Set $s=\frac{1}{2p}$. From Lemma \ref{lemmagreen3} we get 
$$0\le (2\alpha_p)^\frac{p-1}{p}  G_s(x,y)\leq I_\frac{1}{p}(x-y)=|x-y|^{\frac1p-1},$$
where $G_s$ is the Green's function of the interval $I$ defined in Lemma \ref{lemmagreen3}.
Choosing $f:=|(-\Delta)^\frac{1}{2p} u|\big|_{I}$ and using \eqref{G<F} and \eqref{eqgreenb}, we bound
$$ (2\alpha_p)^\frac{p-1}{p} |u(x)| \le   (2\alpha_p)^\frac{p-1}{p} \int_I G_s(x,y)f(y)dy \le I_\frac{1}{p}*f(x) $$
and \eqref{stimaMT} follows at once from \eqref{stimaMT2bis}.

\medskip

It remains to show \eqref{stimaMT2s}. The proof is based on the construction of suitable test functions and it is split into steps.

\medskip

\noindent \emph{Step 1. Definition of the test functions.}
We fix $\tau\ge 1$ and  set
\begin{equation}\label{deff}
f(y)=f_\tau (y):=\frac{1}{2\tau}|y|^{-\frac{1}{p}}\chi_{[-\frac12,-r]\cup [r,\frac12]},\quad r:=\frac{e^{-\tau}}{2}.
\end{equation}
Notice that
$$\|f\|_{L^{p}}^{p}=\frac{2}{(2\tau)^p}\int_r^{\frac12}\frac{dy}{y}=\frac{1}{(2\tau)^{p-1}}.$$
Now let $u=u_\tau\in \tilde H^{s,2}(I)$ solve
\begin{equation}\label{equf}
\left\{
\begin{array}{ll}
(-\Delta)^s u=f &\text{in }I\\
u\equiv 0&\text{in }I^c.
\end{array}
\right.
\end{equation}
in the sense of Theorem \ref{trmexist} in the appendix. 

\medskip

\noindent\emph{Step 2. Proving that $u\in \tilde H^{2s,p}(I)$.}
According to  Proposition \ref{trmbordo} $u$ satisfies
\begin{equation}\label{ubordo}
|u(x)|\le C\|f\|_{L^\infty} (1-|x|)^s\quad \text{for }x\in I.
\end{equation}
We want to prove that $(-\Delta)^s u\in L^p(\R{})$. Since by Proposition \ref{lapint}
$$(-\Delta)^s u(x)= C_s\int_I\frac{-u(y)}{|x-y|^{1+2s}}dy,\quad \text{for }|x|>1$$
and $u$ is bounded, we see immediately that
$$|(-\Delta)^s u(x)|\le \frac{C}{|x|^{1+2s}},\quad \text{for }|x|\ge 2,$$
hence
\begin{equation}\label{stimadeltau1}
\| (-\Delta)^s u\|_{L^q(\R{}\setminus [-2,2])}<\infty\quad \text{for every } q\in [1,\infty).
\end{equation}
Now we claim that
\begin{equation}\label{stimadeltau2}
(I):=\|(-\Delta)^s u\|_{L^q([-2,2]\setminus [-1,1])}<\infty, \quad q=\max\{p,2\}.
\end{equation}
Again using Proposition \ref{lapint}, \eqref{ubordo} and translating, we have
$$(I)=\left(\int_{[-2,2]\setminus [-1,1]}\left|C\int_{-1}^1\frac{-u(y)dy}{|y-x|^{1+2s}}\right|^q dx\right)^\frac{1}{q}\le C\left(\int_{-1}^0\left|\int_0^2\frac{y^sdy}{(y-x)^{1+2s}}\right|^q dx\right)^\frac1q,$$
and using the Minkowski inequality
$$\left(\int_{A_1}\left|\int_{A_2}F(x,y)dy\right|^q dx\right)^\frac1q\le \int_{A_2}\left(\int_{A_1} |F(x,y)|^qdx\right)^\frac1q dy,$$
we get
\[
(I)\le C\int_0^2y^s\left(\int_{-1}^0\frac{dx}{(y-x)^{(1+2s)q}}\right)^\frac1q dy
\le C\int_0^2\frac{dy}{y^{1+s-\frac1q}}<\infty,
\]
since $1+s-\frac1q<1.$ This proves \eqref{stimadeltau2}.

To conclude that $(-\Delta)^s u\in L^p(\R{})$ it remains to show that $(-\Delta)^s u$ does not concentrate on $\de I=\{-1,1\}$, in the sense that the distribution defined by
\[\begin{split}
\langle T,\varphi\rangle&:=\int_{\R{}} u(-\Delta)^s\varphi dx -\int_{I} f\varphi dx -C_s\int_{I^c} \int_{\R{}} \frac{-u(y)}{|x-y|^{1+2s}}dy\, \varphi(x) dx\\
&=:\langle T_1,\varphi\rangle-\langle T_2,\varphi\rangle-\langle T_3,\varphi\rangle \quad \text{for }\varphi\in C^\infty_c(\R{})
\end{split}\]
vanishes. Notice that $\langle T,\varphi\rangle=0$ for $\varphi\in C^\infty_c(\R{}\setminus \de I)$, since $T_1=(-\Delta)^s u$, while
$$\langle T_2,\varphi\rangle =\langle (-\Delta)^s u,\varphi\rangle,\quad  \langle T_3,\varphi\rangle =0 \quad \text{for  }\varphi\in C^\infty_c(I)$$
by \eqref{equf}, and
$$\langle T_2,\varphi\rangle =0,\quad  \langle T_3,\varphi\rangle =\langle (-\Delta)^s u,\varphi\rangle \quad \text{for  }\varphi\in C^\infty_c(I^c)$$
by Proposition \ref{lapint}, and for $\varphi\in C^\infty_c(\R{}\setminus \de I)$ we can split $\varphi=\varphi_1+\varphi_2$ with $\varphi_1\in C^\infty_c(I)$ and $\varphi_2\in C^\infty_c(I^c)$.
In particular $\mathrm{supp}(T)\subset \de I.$

It is easy to see that $T_1$ is a distribution of order at most $1$, i.e.
$$\left|\int_{\R{}} u (-\Delta)^s\varphi dx\right|\le C\|\varphi\|_{C^1(\R{})},\quad \text{for every }\varphi\in C^\infty_c(\R{})$$
(use for instance Proposition \ref{lapint}), and that $T_2$ and $T_3$ are distributions of order zero, i.e. 
$$|\langle T_i ,\varphi\rangle|\le C\|\varphi\|_{L^\infty(\R{})}\quad \text{for }i=2,3.$$
Since $\mathrm{supp} (T)\subset \de I$ it follows from Schwartz's theorem (see e.g. \cite[Sec. 6.1.5]{bon})  that
$$T=\alpha\delta_{-1}+\beta \delta_{1}+\tilde \alpha D\delta_{-1}+\tilde \beta D\delta_1,\quad \text{for some }\alpha,\beta,\tilde \alpha,\tilde\beta \in\R{},$$
where $\langle D\delta_{x_0},\varphi\rangle :=-\langle \delta_{x_0},\varphi'\rangle=-\varphi'(x_0)$ for $\varphi\in C^\infty_c(\R{})$.

In order to show that $\tilde \alpha =0$, take $\varphi\in C^\infty_c(\R{})$ with
$$\mathrm{supp}(\varphi)\subset (-1,1),\quad \varphi'(0)=1,\quad \varphi(0)=0,$$
and rescale it by setting for
$\varphi_\lambda(-1+x)= \lambda \varphi(\lambda^{-1}x)$ for $\lambda>0$.
Since $T_2$ and $T_3$ have order $0$ it follows
$$|\langle T_i,\varphi_\lambda\rangle|\le C \lambda \to 0 \text{ as }\lambda\to 0,\quad \text{for }i=2,3.$$
As for $T_1$, using Proposition \ref{lapint} we get 
\[\begin{split}
\frac{\langle T_1,\varphi_\lambda\rangle}{C_s}&= \int_{(B_{2\lambda}(-1))^c}u(x)\int_{B_\lambda(-1)}\frac{-\varphi_\lambda(y)}{|x-y|^{1+2s}}dydx\\
& \quad +\int_{B_{2\lambda}(-1)}u(x)\int_{(B_{4\lambda}(-1))^c}\frac{\varphi_\lambda(x)}{|x-y|^{1+2s}}dydx\\
&\quad +\int_{B_{2\lambda}(-1)}u(x)\int_{B_{4\lambda}(-1)}\frac{\varphi_\lambda(x)-\varphi_\lambda(y)}{|x-y|^{1+2s}}dydx\\
&=: (I)+(II)+(III).
\end{split}\]
Since $\|\varphi_\lambda\|_{L^\infty(\R{})}=C_\varphi\lambda$ and $u\in L^\infty(\R{})$, one easily bounds $|(I)|+|(II)|\to 0$ as $\lambda\to 0$, and using that $\sup_{\R{}}|\varphi'_\lambda|=\sup_{\R{}}|\varphi'|$ we get
$$|(III)|\le \int_{B_{2\lambda}(-1)}|u(x)|\int_{B_{4\lambda}(-1)}\frac{\sup_{\R{}}|\varphi'|}{|x-y|^{2s}}dydx\le C\lambda^{1-2s}\int_{B_{2\lambda}(-1)}|u(x)|dx\to 0\quad \text{as }\lambda\to 0.$$
Since for $\lambda\in (0,1)$ we have $\langle T,\varphi\rangle=- \tilde \alpha$, by letting $\lambda\to 0$ it follows that $\tilde\alpha=0$. Similarly one can prove that $\tilde\beta=0$.

We now claim that $\alpha,\beta=0$.
Considering
$$\tilde u(x):= u(x)-\alpha F_s(x +1)-\beta F_s(x -1),$$
and recalling that $(-\Delta)^s F_s=\delta_0$, one obtains that
$$(-\Delta)^s\tilde u= T_1-\alpha\delta_{-1}-\beta\delta_{1}=T_2+T_3 \in L^2(\R{}),$$
hence with Proposition \ref{HW}
$$\int_{\R{}}\int_{\R{}}\frac{|\tilde u(x)-\tilde u(y)|^2}{|x-y|^{1+2s}}dxdy= [\tilde u]_{W^{2s,2}(\R{})}^2=C\|(-\Delta)^s \tilde u\|_{L^2(\R{})}^2<\infty,$$
and this gives a contradiction if $\alpha\neq 0$ or $\beta\ne 0$ since the integral on the left-hand side does not converge in these cases.

Then  $T=0$, i.e. $(-\Delta)^s u=:T_1= T_2+T_3$ and from \eqref{equf}, \eqref{stimadeltau1} and \eqref{stimadeltau2} we conclude that $(-\Delta)^s u\in L^p(\R{})$, hence $u\in \tilde H^{2s,p}(I)$, as wished.

\medskip

\noindent\emph{Step 3: Conclusion.}
Recalling that $(-\Delta)^s u=f$ in $I$, from \eqref{eqgreenb} we have for $x\in I$
\begin{equation}\label{uu1u2}
\begin{split}
u(x)&=\int_{I}G_s(x,y)f(y)dy\\
&=\frac{1}{2\tau(2\alpha_p)^\frac{p-1}{p}}\int_{r<|y|<\frac12} \frac{1}{|x-y|^{1-\frac{1}{p}}|y|^{\frac{1}{p}}}dy-\int_{r<|y|<\frac12} H_s(x,y)f(y)dy\\
&=:u_1(x)+u_2(x),
\end{split}
\end{equation}
where $H_s(x,y)$ is as in Lemma \ref{lemmagreen3}.

We now want a lower bound for $u$ in the interval $[-r,r]$. We fix $0<x\le r$ and estimate
\[\begin{split}
u_1(x)&=\frac{1}{2\tau(2\alpha_p)^\frac{p-1}{p}}\left( \int_r^\frac12 \frac{dy}{(y-x)^{1-\frac1p}y^{\frac1p}}+\int^{-r}_{-\frac12} \frac{dy}{|y-x|^{1-\frac1p}|y|^{\frac1p}}\right)\\
&\ge \frac{1}{2\tau(2\alpha_p)^\frac{p-1}{p}}\left(\int_r^\frac12 \frac{dy}{y}+\int^{\frac 12}_{r}\frac{dy}{y+x}\right)\\
&=\frac{1}{2\tau(2\alpha_p)^\frac{p-1}{p}}\left(2\tau+\log\left(\frac{1+2x}{1+\frac{x}{r}}\right) \right)\\
&=\frac{1}{(2\alpha_p)^\frac{p-1}{p}}+O(\tau^{-1}).
\end{split}\]
Since $H_s$ is bounded on $[-r,r]\times [-\tfrac12,\tfrac12]$, we have
$$|u_2(x)|\leq C \int_{r}^{\frac12} f(y)dy \leq C\tau^{-1} \int_{0}^{\frac12} |y|^{-\frac{1}{p}}dy=O(\tau^{-1}),\quad x\in [-r,r].$$
Then
$$u=u_\tau \ge  \frac{1}{(2\alpha_p)^\frac{p-1}{p}}+O(\tau^{-1})\quad \text{on }[-r,r],$$
as $\tau\to\infty$. We now set
$$w_\tau:=(2\tau)^{\frac{p-1}{p}}u_\tau\in \tilde H^{\frac{1}{p},p}(I),$$
so that $\|(-\Delta)^s w_\tau\|_{L^p(I)}=1$,
we compute 
\[
\int_I  e^{\alpha_p  |w_\tau|^{p'}}dx\ge\int_{-r}^r e^{\tau+O(1)} dx\ge \frac{2r e^\tau}{C} =\frac{1}{C},
\]
and using that $\inf_{[-r,r]} w_\tau \to\infty$ as $\tau\to\infty$, we conclude
$$\lim_{\tau\to\infty } \int_I h(w_\tau)  e^{\alpha_p  |w_\tau|^{p'}}dx =\infty,$$
whenever $h$ satisfies $\lim_{t\to\infty} h(t)=\infty.$
\hfill$\square$

\subsection{A few consequences of Theorem \ref{MT2}}

\begin{lemma}\label{Lp}
Let $u\in H$. Then $u^{q}e^{pu^{2}}\in L^{1}(I)$ for every $p,q>0$. 
\end{lemma}

\begin{proof}
Since $|u|^q\le C(q)e^{|u|^{2}}$, it is enough to prove the case $q=0$. Given $\varepsilon>0$ (to be fixed later), by Lemma \ref{lemmadens} there exists $v\in C^{\infty}_{c}(I)$ such that  
$$\|v-u\|_{H}^{2}<\varepsilon.$$
Using 
$$u^2\leq (v-u)^2+2vu$$
we bound
\begin{equation}\label{eqlemmaLp}
e^{p u^2}\leq e^{p(v-u)^2}e^{2pvu}.
\end{equation}
Using the inequality $|ab|\leq \frac{1}{2}(a^{2}+b^{2})$ we have

$$e^{2puv}\leq e^{\frac{1}{\varepsilon}p^{2}\|u\|_{H}^{2}v^{2}}e^{\varepsilon(\frac{u}{\|u\|_{H}})^{2}},$$
and for $\ve$ small enough the right-hand side is bounded in $L^2(I)$ thanks to Theorem \ref{MT2}. Still by Theorem \ref{MT2} we have $e^{p(v-u)^{2}}\in L^2(I)$ if $\ve>0$ is small enough, hence going back to \eqref{eqlemmaLp} and using that $v\in L^\infty(I)$ is now fixed, we conclude with H\"older's inequality that $e^{pu^2}\in L^1(I)$.
\end{proof}

\begin{lemma}\label{con}
For any $q,p\in (1,+\infty)$ the functional
\begin{equation*}
E_{q,p}: H \to \R{},\quad E_{q,p}(u):=\int_{I}|u|^{q}e^{pu^{2}}dx
\end{equation*}
is continuous.
\end{lemma}

\begin{proof}
Consider a sequence $u_k\to u$ in $H$. By Lemma \ref{Lp} (up to changing the exponents) we have that the sequence $f_k:= |u_k|^qe^{p u_k ^2}$ is bounded in $L^2(I)$. Indeed, it is enough to write $u_k=(u_k-u)+u$ and use the same estimates as in \eqref{eqlemmaLp} with $u$ instead of $v$ and $u_k$ instead of $u$. 
We now claim that $f_k\to f$ in $L^1(I)$. Indeed up to a subsequence $u_k\to u$ a.e., hence $f_k\to f:=|u|^qe^{pu^2}$ a.e. 
Then considering that since $f_k$ is bounded in $L^2(I)$ we have
$$\int_{\{f_k>L\}}f_k\, dx\leq \frac 1L \int_{\{f_k>L\}}f_k^2\, dx\leq \frac CL \to 0\quad\text{as $L\to+\infty$}, $$
the claim follows at once from Lemma \ref{lemmafkL}.
\end{proof}

\begin{lemma}\label{lemmaC1}
The functional $J: H\to\R{}$ defined in \eqref{defJ} is of class $C^\infty$.
\end{lemma}

\begin{proof} 
This follows easily from Lemma \ref{con}, since the first term on the right-hand side of \eqref{defJ} is simply $\frac{1}{2}\|u\|_H^2$, and the derivatives of the second term are continuous thanks to Lemma \ref{con}. The details, at least to prove that $J\in C^1(H)$, are essentially as in the proof of Lemma 2.1 of \cite{str}. The higher-order differentials are handled in the same way since they have a similar form, with the non-linear term $e^{\frac12 u^2}$ just multiplied by polynomial terms.
\end{proof}

The following lemma is a fractional analog of a well-known result of P-L. Lions \cite{PLL}. 

\begin{lemma}\label{lemmaPLL} Consider a sequence $(u_k)\subset H$ with $\|u_{k}\|_H=1$ and $u_k\rightharpoonup u$ weakly in $H$, but not strongly (so that $\|u\|_H<1$).  Then if $u\not \equiv 0$,  $e^{\pi u_{k}^{2}}$ is bounded in $L^{p}$ for $1\le p<\tilde p:=(1-\|u\|^{2}_H)^{-1}$.
\end{lemma}

\begin{proof}
We split
$$u_{k}^{2}=u^2 -2u(u-u_{k})+(u-u_{k})^{2}.$$
Then $v_{k}:=e^{\pi u_{k}^{2}}=vv_{k,1}v_{k,2}$, where $v=e^{\pi|u|^{2}} \in L^{p}(I)$  for all $p\geq 1$ by Lemma \ref{Lp}, $v_{k,1}=e^{-2\pi u(u-u_{k})}$ and $v_{k,2}=e^{\pi (u-u_{k})^{2}}$.\\
Notice now that from 
$$ -2p\pi u(u-u_{k})\leq \pi \left(\frac{p^{2}}{\varepsilon ^{2}}u^{2}+\varepsilon^{2} (u-u_{k})^{2}\right),$$
we get from Lemma \ref{Lp} and Theorem \ref{MT2} that $v_{k,1}\in L^{q}(I)$ for all $q\geq 1$ if $\ve>0$ is small enough (depending on $q$). But again from Theorem \ref{MT2} $v_{2,k}$ is bounded in $ L^{p}(I)$ for all $p<\tilde p$ since
$$\|u_{k}-u\|^{2}_H=1-2\langle u_{k},u\rangle+\|u\|_H^{2}\to 1-\|u\|_H^{2}.$$
Therefore by H\"older's inequality we have that $v_{k}$ is bounded in $L^{p}(I)$ for all $p<\tilde p$.
\end{proof}

\section{Proof of Proposition \ref{trmex}}\label{PSC}

For the proof of Proposition \ref{trmex} we will closely follow \cite{A}.
Set
\begin{equation}\label{defI}
Q(u):=J(u)-\frac{1}{2}\langle J'(u),u\rangle=\lambda\int_{I} \left(\left(\frac{u^{2}}{2}-1\right)e^{\frac{1}{2} u^{2}}+1\right) dx.
\end{equation}
\begin{rmk}\label{rmkI} Notice that the integrand on the right-hand side of \eqref{defI} is strictly convex and has a minimum at $u=0$; in particular
\begin{equation}\label{minI}
0=Q(0)<Q(u) \quad \text{for every }u\in H\setminus\{0\}.
\end{equation}
Furthermore by Lemma \ref{con} the functional $Q$ is continuous on $H$ and by convexity $Q$ is also weakly lower semi-continuous.
\end{rmk}

Let us also notice that
\begin{equation*}
\begin{split}
\lambda\int_{I}u^2e^{\frac{1}{2}u^2}\, dx&=\lambda\int_{\{|u|\le 4\}}u^2e^{\frac{1}{2}u^2}\, dx+\lambda\int_{\{|u|>4\}}u^2e^{\frac{1}{2}u^2}\, dx\\&\le C+\lambda\int_{\{|u>4|\}}u^2e^{\frac{1}{2}u^2}\, dx\le C+\tilde C Q(u)
\end{split}
\end{equation*}
and hence we have 
\begin{equation}\label{ineq1}
\lambda\int_{I} u^{2}e^{\frac{1}{2} u^{2}}dx \leq C(1+Q(u))\quad \text{for every }u\in H.
\end{equation}

We consider a Palais-Smale sequence $(u_{k})_{k\geq 0}$ with $J(u_{k})\to c$. From \eqref{condPS} we get
$$\langle J'(u_k),u_k\rangle = o(1)\|u_k\|_{H}\quad \text{as } k\to\infty,$$
and
\begin{equation}\label{Ibound}
Q(u_{k})= J(u_k)-\frac{1}{2}\langle J'(u_k),u_k\rangle = c+o(1)+ o(1)\|u_k\|_{H}.
\end{equation}
Then with \eqref{ineq1} we have
$$\lambda\int_{I} u_{k}^{2}e^{\frac{1}{2} u^{2}_{k}}dx\le C\left(1+\|u_{k}\|_{H} \right),$$
hence, using that $Q(u_k)\ge 0$
$$\lambda\int_{I} \left(e^{\frac{1}{2}u_{k}^{2}}-1\right) dx \le C\left(1+\|u_{k}\|_{H} \right),$$
so that
$$J(u_k) \ge \frac{1}{2}\|u_k\|_{H}^2 - C (1+\|u_k\|_{H}).$$
This and the boundedness of $(J(u_k))_{k\ge 0}$ yield that the sequence $(u_{k})_{k\geq 0}$ is bounded in $H$, hence we can extracts a weakly converging subsequence $u_k\rightharpoonup \tilde u$ in $H$.  By the compactness of the embedding $H\hookrightarrow L^2$ (see e.g. \cite[Theorem 7.1]{DNPV}, which we can apply thanks to \cite[Proposition 3.6]{DNPV}, see Proposition \ref{HW}), up to extracting a further subsequence we can assume that $u_k\to \tilde u$ almost everywhere. To complete the proof of the theorem it remains to show that, up to extracting a further subsequence, $u_k\to \tilde u$ strongly in $H$.

\medskip

By Remark \ref{rmkI} we have
\begin{equation}\label{limIk}
0\leq Q(\tilde u) \leq \liminf_{k\to \infty} Q(u_{k})
= \liminf_{k\to \infty} \left(J(u_{k})-\frac{1}{2}\langle J'(u_{k}),u_{k}\rangle\right)
=c
\end{equation}
Thus necessarily $c\ge 0$. In other words the Palais-Smale condition is vacantly true when $c<0$ because no sequence can satisfy \eqref{condPS}.

\medskip

Clearly \eqref{limIk} implies $Q(u_{k})\to Q(\tilde u)=0$. We now claim that
\begin{equation}\label{convL1}
u_{k}^pe^{\frac{1}{2}u_{k}^{2}}\to \tilde u^p e^{\frac{1}{2}\tilde u^2}\quad \text{in } L^{1}(I)\quad \text{for }0\le p<2.
\end{equation}
Indeed, up to extracting a further subsequence, from \eqref{ineq1} and \eqref{limIk} we get
$$\int_{\{|u_{k}|>L\}}u_{k}^p e^{\frac{1}{2}u_{k}^{2}}dx \leq \frac{1}{L^{2-p}}\int_{\{|u_{k}|>L\}}u^{2}_{k}e^{\frac{1}{2}u_{k}^{2}}dx=O\left(\frac{1}{L^{2-p}}\right),$$
and \eqref{convL1} follows from Lemma \ref{lemmafkL} in the appendix. 

Let us now consider the case $c=0$. 
Since $Q(\tilde u)=0$, hence $\tilde u\equiv 0$, with \eqref{convL1} we get
\begin{equation}\label{PS1}
\lim_{k\to\infty} \|u_{k}\|_{H}^{2}=2\lim_{k\to\infty}\left( J(u_{k})+\lambda\int_{I}\left(e^{\frac{1}{2}u_{k}^{2}}-1\right) dx\right)
=2\lambda\int_{I}\left(e^{\frac{1}{2}\tilde{u}^{2}}-1\right) dx=0,
\end{equation}
so that $u_k\to 0$ is $H$ and the Palais-Smale condition holds in the case $c=0$ as well.

\medskip

The last case is when $c\in(0, \pi )$. We will need the following result which is analogue to Lemma 3.3 in \cite{A}.
\begin{lemma}\label{comp}
Consider a bounded sequence $(u_k)\subset H$ such that $u_{k}$ converges weakly and almost everywhere to a function $u\in H$. Further assume that:
\begin{enumerate}
\item there exists $c \in (0, \pi]$ such that $J(u_{k})\to c$;
\item $\|u\|_{H}^{2}\geq \lambda\int_{I} u^{2}e^{\frac{1}{2}u^2} dx$;
\item $\sup_{k} \int_{I}u_{k}^{2}e^{\frac{1}{2}u_{k}^2}dx <\infty$;
\item either $u\not\equiv 0$ or $c<\pi$.
\end{enumerate}
Then $$\lim_{k\to \infty} \int_{I}u_{k}^{2}e^{\frac{1}{2}u_{k}^2}dx = \int_{I}u^{2}e^{\frac{1}{2}u^2}dx.$$
\end{lemma}
\begin{proof}
We assume $u\not \equiv 0$ (if $u\equiv 0$ and $c<\pi$ the existence of $\ve>0$ in \eqref{1+eps} below is obvious). We then have $Q(u)>0$. On the other hand from assumption 2 we get
$$J(u)=\frac{1}{2}\|u\|_{H}^{2}+Q(u)-\frac\lambda2 \int_{I} u^{2}e^{\frac{1}{2}u^2} dx\geq Q(u)>0.$$
We also know from the weak convergence of $u_k$ to $u$ in $H$, the weakly lower semicontinuity of the norm and \eqref{convL1}  that
$$J(u)\le \lim_{k\to\infty} J(u_k)=c,$$
where the inequality is strict, unless $u_k\to u$ strongly in $H$ (in which case the proof is complete).
Then one can choose $\varepsilon>0$ so that 
\begin{equation}\label{1+eps}
 \frac{1+2\varepsilon}{\pi}<\frac{1}{c-J(u)}.
\end{equation}
Notice now that if we set $\beta=\lambda\int_{I} \left(e^{\frac{1}{2}u^{2}}-1\right)dx$, then
$$\lim_{k\to\infty}\|u_{k}\|_{H}^{2}=2c+2\beta.$$
Then multiplying \eqref{1+eps} by $\frac12 \|u_k\|_{H}^2$ we have for $k$ large enough
$$\frac{1+\varepsilon}{2\pi}\|u_{k}\|_{H}^{2}\le \tilde p:= \frac{1+2\varepsilon}{2\pi}\lim_{k\to\infty}\|u_{k}\|_{H}^{2}  <  \frac{c+\beta}{c-J(u)}=\left(1-\frac{\|u\|_{H}^{2}}{2(c+\beta)}\right)^{-1}.$$
By Lemma \ref{lemmaPLL} applied to $v_k:=\frac{u_{k}}{\|u_{k}\|_{H}}$, we get that the sequence $\exp(\tilde p\pi v_k^2)$ is bounded in $L^{1}(I)$, hence $e^{\frac{(1+\varepsilon)}{2}u^{2}_{k}}$ is bounded in $L^{1}$.\\
Now we have that 
\[
\begin{split} 
\int_{\{|u_{k}|>K\}}u_{k}^{2}e^{\frac{1}{2}u_{k}^{2}}dx&=\int_{\{|u_{k}|>K\}}\left(u_{k}^{2}e^{-\frac{\varepsilon}{2}u_{k}^{2}}\right) e^{\frac{1+\varepsilon}{2}u_{k}^{2}}dx\\
&\leq o(1) \int_{\{|u_{k}|>K\}} e^{\frac{1+\varepsilon}{2}u_{k}^{2}}dx
\end{split}
\]
with $o(1)\to 0$ as $K\to\infty$, and we conclude with Lemma \ref{lemmafkL}.
\end{proof}

We now claim 
\begin{equation}\label{tildeuN}
\|\tilde u\|_{H}^2=\lambda\int_I \tilde u^2e^{\frac12 \tilde u^2}dx.
\end{equation}
First we show that $\tilde{u}\not\equiv 0$. So for the sake of contradiction, we assume that $\tilde{u}\equiv 0$. By Lemma \ref{comp}
$$\lim_{k\to \infty} \int_{I}u_{k}^{2}e^{\frac{1}{2}u_{k}^{2}}dx=0.$$
Therefore, also using \eqref{convL1}, we obtain $\lim_{k\to \infty} Q(u_{k})=0$. It follows that
$$0<c=\lim_{k\to\infty} J(u_{k})=\lim_{k\to\infty} \left (Q(u_{k})+\frac{1}{2}\langle J'(u_{k}),u_{k}\rangle\right)=0,$$
contradiction, hence $\tilde{u}\not \equiv 0$.

Fix now $\varphi \in C^{\infty}_{0}(I)\cap H$. We have $\langle J'(u_{k}),\varphi\rangle\to 0$ as $k\to \infty$, since $(u_k)$ is a Palais-Smale sequence. But, by weak convergence we have that 
$$(u_k,\varphi)_H \to (\tilde u,\varphi)_H.$$
Now \eqref{convL1} implies 
$$\int_{I}\varphi u_{k}e^{\frac{1}{2}u_{k}^{2}}dx \to \int_{I}\varphi \tilde{u}e^{\frac{1}{2}\tilde{u}^{2}}dx,\quad \text{for every }\varphi \in C^{\infty}_{0}(I).$$
Thus we have  
$$(\tilde u,\varphi)_H=\lambda\int_{I}\varphi \tilde{u}e^{\frac{1}{2}\tilde{u}^{2}}dx.$$
By density and the fact that $\tilde{u}e^{\frac{1}{2}\tilde{u}^{2}}\in L^{p}$ for all $p\geq 1$, we have that
$$(\tilde u,\tilde u)_H=\lambda \int_{I} \tilde{u}^{2}e^{\frac{1}{2}\tilde{u}^{2}}dx,$$
hence \eqref{tildeuN} is proven. Therefore, we are under the assumptions of Lemma \ref{comp}, which yields
\begin{equation}
\begin{split}
\|\tilde{u}\|_{H}^{2}&\leq \liminf_{k\to\infty}\|u_{k}\|_{H}^{2}\\
&=2\liminf_{k\to\infty}\left[ J(u_{k})+\lambda\int_{I} \left(e^{\frac{1}{2}u_{k}^{2}}-1\right) dx\right] \\
&=2\liminf_{k\to\infty}\left[ \frac\lambda 2\int_{I}u_{k}^{2}e^{\frac{1}{2}u_{k}^{2}} dx+\frac{1}{2}\langle J'(u_{k}),u_{k}\rangle\right]\\
&=\lambda \int_{I}\tilde u^{2}e^{\frac{1}{2}\tilde u^{2}} dx\\
&=\|\tilde{u}\|_{H}^{2}.
\end{split}
\end{equation}
By Hilbert space theory the convergence of the norms implies that $u_k\to \tilde u$ strongly in $H$, and the Palais-Smale condition is proven.


\section{Proof of Theorem \ref{ex}}\label{CP}

We start by proving the last claim of Theorem \ref{ex}.

\begin{prop}\label{nonex}
Let $u$ be a non-negative non-trivial solution to \eqref{eq0} for some $\lambda\in \R{}$. Then $0<\lambda<\lambda_1(I)$.
\end{prop}

\begin{proof}
Let $\varphi_{1}\ge 0$ be as in Lemma \ref{phi1}. Then using $\varphi_1$ as a test function in \eqref{eq0} (compare to \eqref{eq1}) yields 
$$\lambda_{1}(I) \int_{I} u\varphi_{1}dx=\lambda \int_{I}u\varphi_{1}e^{\frac{1}{2}|u|^{2}}dx> \lambda \int_{I} u\varphi_1dx.$$
Hence $\lambda<\lambda_{1}$. Using $u$ as test function in \eqref{eq0} gives at once $\lambda>0$.
\end{proof}

The rest of the section is devoted to the proof of the existence part of Theorem \ref{ex}. 

\medskip

Define the Nehari manifold
$$N(J):=\left\{u\in H\setminus \{0\}; \langle J'(u),u\rangle=0\right\}.$$
Since, according to \eqref{defI}-\eqref{minI}, $J(u) = Q(u)>0$ for $u\in N(J)$, we have
$$a(J):=\inf_{u\in N(J)} J(u)\ge 0.$$

\begin{lemma}\label{lemmaaJ}
We have $a(J)>0$.
\end{lemma}
\begin{proof}
Assume that $a(J)=0$, then there exists a  sequence $(u_{k})\subset N(J)$ such that 
$$J(u_{k})=Q(u_k)\to 0\quad\text{ as $k\to \infty$,}$$ 
From (\ref{ineq1}) we infer
\begin{equation}\label{stimaquad}
\sup_{k\ge 0} \int_{I}u_{k}^{2}e^{\frac{1}{2}u_{k}^{2}}dx <\infty,
\end{equation}
which, again using the fact that $u_{k}\in N(J)$, implies that $\|u_{k}\|_{H}$ is bounded. Thus, up to extracting a subsequence, we have that $u_{k}$ weakly converges to a function $u\in H$. From the weak lower semicontinuity of $Q$ we then get
$$0\leq Q(u) \leq \liminf_{k\to\infty} Q(u_{k})=0,$$
thus $J(u)=Q(u)=0$ and \eqref{minI} implies $u\equiv 0$.
On the other hand, we have from \eqref{PS1} with $\tilde u$ replaced by $u$ (which holds with the same proof thanks to \eqref{stimaquad})
\begin{equation}\label{uk0}
\lim_{k\to \infty} \|u_{k}\|_{H}^{2}=2\lim_{k\to \infty} \left\{J(u_{k}) + \lambda \int_I \left(e^{\frac{1}{2}u_{k}^{2}}-1 \right) dx\right\} =0,
\end{equation}
therefore we have strong convergence of $u_{k}$ to $0$.

Now, if we let $v_{k}=\frac{u_{k}}{\|u_{k}\|_{H}}$ and up to a subsequence we assume $v_k\to v$ weakly in $H$ and almost everywhere, we have
\begin{equation}\label{1<1}
1=\|v_k\|_{H}^2=\lim_{k\to \infty} \lambda \int_{I}e^{\frac{1}{2}u_{k}^{2}}v_{k}^{2}dx = \lambda \int_{I} v^{2} dx < \lambda_{1}\int_{I} v^{2} dx \leq 1,
\end{equation}
where the third equality is justified as follows: From the Sobolev imbedding $v_{k}\to v$ in all $L^{p}(I)$ for every  $ p\in [1,\infty)$, while from $\eqref{uk0}$ and Theorem \ref{MT2} we have that for every $q\in [1,\infty)$ the sequence $(e^{\frac{1}{2}u_{k}^{2}} )$ is bounded in $L^q(I)$, hence from H\"older's inequality we have the desired limit. The last inequality in \eqref{1<1} follows from the Poincar\'{e} inequality, see \eqref{poinc}.

Clearly \eqref{1<1} is a contradiction, hence $a(J)>0$.
\end{proof}

\begin{lemma}\label{nihari}
For every $u\in H\setminus \{0\}$ there exists a unique $t=t(u)>0$ such that $t u \in N(J)$.
 Moreover, if 
\begin{equation}\label{eqnihari}
\|u\|_{H}^{2}\leq \lambda \int_{I} u^{2}e^{\frac{1}{2}u^{2}}dx,
\end{equation}
then $t(u) \leq 1 $ and $t(u)=1$ if and only if $u\in N(J)$.
\end{lemma}

\begin{proof}
Fix $u\in H\setminus\{0\}$ and for $t\in (0,\infty)$ define the function
$$f(t)=t^{2}\left(\|u\|_{H}^{2}-\lambda \int_{I}u^{2}e^{\frac{1}{2}t^{2}u^{2}}dx\right),$$
which can also be written as
$$f(t)=t^{2}\left(\|u\|^{2}_{H}-\lambda\int_{I}u^{2}dx \right) -  t^{2}\lambda \int_{I}u^{2}\left(e^{\frac{1}{2}t^{2}u^{2}}-1\right)dx.$$
Notice that $tu\in N(J)$ if and only if $f(t)=0$. 

From the inequality $$u^{2}\left(e^{\frac{1}{2}t^{2}u^{2}}-1\right)\geq \frac 12 t^{2}u^{4}$$
we infer 
$$f(t)\leq t^{2}\left(\|u\|_{H}^{2}-\lambda \int_{I}u^{2}dx \right) - \frac 12 t^4\lambda \int_{I}u^{4}dx,$$
hence
$$\lim_{t\to +\infty} f(t)=-\infty.$$

Now notice that the function $t\mapsto\left(e^{\frac{1}{2}t^{2}u^{2}}-1\right)$ is monotone increasing on $(0,\infty)$, and by Lemma \ref{Lp} we have $\left(e^{\frac{1}{2}u^{2}}-1\right)\in L^{p}(I)$ for all $p\in [1,\infty)$, so that 
$$u^2\left(e^{\frac{1}{2}u^{2}}-1\right)\in L^{1}(I).$$ 
Then by the dominated convergence theorem we get
$$\lim_{t\to 0} \int_{I}u^{2}\left(e^{\frac{1}{2}t^{2}u^{2}}-1\right)dx=0.$$
So one has
$$f(t)=t^{2}\left(\|u\|_{H}^{2}-\lambda \int_{I}u^{2}dx\right) +o(t^{2})\quad \text{as }t\to 0.$$
Hence, $f(t)>0$  for $t$ small, since for $\lambda<\lambda_1(I)$
$$\|u\|_{H}^{2}-\lambda \int_{I}u^{2}dx >0$$
(compare the proof of Lemma \ref{phi1}).
Therefore there exists $t=t(u)$ such that $f(t)=0$, i.e. $tu\in N(J)$. The uniqueness of such $t$ follows noticing that the function
$$t\mapsto \int_{I} u^{2}e^{\frac{1}{2}t^{2}u^{2}}dx$$
is increasing. Keeping this in mind, if we assume \eqref{eqnihari}, then $f(1)\leq 0$, hence $f(t)\leq 0$ for all $t\geq 1$. This implies at once that $t(u)\leq 1$ 
and $t(u)=1$ if and only if $u\in N(J)$.
\end{proof}

\begin{lemma}\label{lemmaaJ2}
We have $a(J)<\pi$.
\end{lemma}
\begin{proof}
Take $w\in H$ such that $\|w\|_{H}=1$ and let $t=t(w)$ be given as in Lemma \ref{nihari} so that $tw\in N(J)$. Then
$$a(J)\leq J(tw)\leq \frac{t^2}{2}\|w\|_H^2=\frac{t^2}{2}.$$
Now using the monotonicity of $t\mapsto \int_{I} w^{2}e^{\frac{1}{2}t^{2}w^{2}}dx$ we have
$$ \lambda \int_{I} w^{2}e^{a(J)w^{2}}dx\leq  \lambda \int_{I} w^{2}e^{\frac{1}{2}t^{2}w^{2}}dx=\frac{t^{2}\|w\|_H^{2}}{t^{2}}=1.$$
Thus
$$\sup_{\|w\|_H=1}\lambda \int_{I} w^{2}e^{a(J) w^{2}}dx \leq 1,$$
and Theorem \ref{MT2}  implies that $a(J)< \pi.$
\end{proof}

\begin{lemma}\label{lemmacrit}
Let $u\in N(J)$ be such that $J'(u)\not=0$, then $J(u)>a(J)$.
\end{lemma}

\begin{proof}
We choose $h\in H$ such that $\langle J'(u),h\rangle=1$, and for $\alpha\in \R{}$ we consider the path $\sigma_{t}(\alpha)=\alpha u -th$, $t\in \R{}$. Remember that by Lemma \ref{lemmaC1} $J\in C^1(H)$. By the chain rule 
$$\frac{d}{dt} J(\sigma_{t}(\alpha))=-\langle J'(\sigma_{t}(\alpha)),h\rangle,$$
therefore, if we set $t=0$, $\alpha=1$ we find
$$\frac{d}{dt}J(\sigma_{t}(\alpha))\bigg|_{t=0,\alpha=1}=-\langle J'(u),h\rangle =-1.$$
Hence there exist $\delta>0$ and $\varepsilon>0$ such that for $\alpha \in [1-\varepsilon,1+\varepsilon]$ and $t \in (0,\delta]$
\begin{equation}\label{Jsigma}
J(\sigma_{t}(\alpha))<J(\sigma_{0}(\alpha))=J(\alpha u).
\end{equation}
Now we consider the function $f$ defined by
$$f_{t}(\alpha)=\|\sigma_{t}(\alpha)\|_{H}^{2}-\lambda \int_{I}\sigma_{t}(\alpha)^{2}e^{\frac{1}{2}\sigma_{t}(\alpha)^{2}}dx,$$
which is continuous with respect to $t$ and $\alpha$ by Lemma \ref{con}.
Notice that since $u\in N(J)$ we have
$$f_{0}(\alpha)=\alpha^{2}\int_{I}u^{2}\left(e^{\frac{1}{2}u^{2}}-e^{\frac{1}{2}\alpha^{2}u^{2}}\right)dx$$
and $f_0(1)=0$. Since the function $\alpha \mapsto u^{2}(e^{\frac{1}{2}u^{2}}-e^{\frac{1}{2}\alpha^{2}u^{2}})$ is decreasing, by continuity we can find $\varepsilon_{1}\in (0,\varepsilon)$ and  $\delta_{1}\in (0,\delta)$  such that
$$f_{t}(1-\varepsilon_{1})>0,\quad  f_{t}(1+\varepsilon_{1})<0\quad \text{for }t\in[0,\delta_{1}].$$
Then if we fix $t\in (0,\delta_1]$ we can find $\alpha_{t}\in [1-\ve_{1},1+\ve_{1}]$ such that $f_t(\alpha_t)=0$, i.e. $\sigma_{t}(\alpha_{t})\in N(J)$, and
from \eqref{Jsigma} we get
$$a(J)\leq J(\sigma_{t}(\alpha_{t}))<J(\alpha_{t}u).$$
Since
$$\frac{d}{d\alpha} J(\alpha u)= f_0(\alpha),$$
and $f_0(\alpha)>0$ for $\alpha<1$ and $f_0(\alpha)<0$ for $\alpha>1$, we get
$$J(\alpha u)\le J(u)\quad \text{for }\alpha\in \R{},$$
and we conclude that
$$a(J)\leq J(\sigma_{t}(\alpha_{t}))<J(\alpha_{t}u)\leq J(u).$$
\end{proof}

\medskip

\noindent\emph{Proof of Theorem \ref{ex} (completed).}
To complete the proof it is enough to show the existence of $u_{0}\in N(J)$ such that $J(u_{0})=a(J)$. We consider then a minimizing sequence $(u_{k})\subset N(J)$. \\
We assume that $u_k$ changes sign. Then since $u_k\in N(J)$ we have
$$\| |u_k|\|^2_{H}< \| u_k\|^2_{H}=\lambda\int_{I}u_k^2e^{\frac{1}{2}u_k^2}dx= \lambda\int_{I}|u_k|^2e^{\frac{1}{2}|u_k|^2}dx,$$
where we used \eqref{|u|}, hence by Lemma \ref{nihari} there exists $t_k=t(|u_k|)<1$ such that $t_k|u_k|\in N(J)$, whence
$$J(t_k|u_k|)=Q(t_k|u_k|)< Q(|u_k|)=Q(u_k)=J(u_k),$$
where the inequality in the middle depends on the monotonicity of $Q$.
Hence up to replacing $u_k$ with $t_k|u_k|$ we can assume that the minimizing sequence (still denoted by $(u_k)$) is made of non-negative functions. 

Since $J(u_k)=Q(u_k)\le C$ we infer from \eqref{ineq1}
$$\int_{I} u_{k}^{2}e^{\frac{1}{2}u_{k}^{2}}dx \le C$$
and for $u_k\in N(J)$ we get
$$\|u_{k}\|_{H}\le C.$$
Thus up to a subsequence $u_{k}$ weakly converges to a function $u_{0}\in H$, and up to a subsequence the convergence is also almost everywhere.

We claim that $u_{0}\not \equiv 0$. Indeed if $u_{0}\equiv 0$, then from (\ref{convL1}), we have that $\left(e^{\frac{1}{2}u_{k}^{2}} -1\right)\to 0$ in $L^{1}(I)$. Thus
$$\lim_{k\to \infty}\|u_{k}\|_{H}^{2}=2\lim_{k\to \infty}\left[ J(u_{k})+\lambda\int_{I}\left(e^{\frac{1}{2}u_{k}^{2}} -1\right) dx\right] =2a(J).$$
Then according to Theorem \ref{MT2}, since $a(J)<\pi$ we have that $e^{\frac{1}{2}u_{k}^{2}}$ is bounded in $L^{p}$ for some $p>1$, hence weakly converging in $L^p(I)$ to $e^{\frac12 u_0^2}$. From the compactness of the Sobolev embeddings (see \cite[Theorem 7.1]{DNPV}, which can be applied thanks to Proposition \ref{HW}), up to a subsequence $u_{k}^{2}\to u_0^2$ strongly in $L^{p'}(I)$, hence $$\lim_{k\to \infty}\int_{I} u_{k}^{2}e^{\frac{1}{2}u_{k}^{2}}dx =\int_{I} u_{0}^{2}e^{\frac{1}{2}u_{0}^{2}}dx=0,$$
and with Lemma \ref{lemmaaJ} and \eqref{defI} one gets
$$0<a(J)=\lim_{k\to \infty} J(u_{k})= \lim_{k\to \infty} Q(u_{k})=0,$$  
which is a contradiction.

Next we claim that 
$$\|u_{0}\|_{H}^{2}\leq \lambda \int_{I}u_{0}^{2}e^{\frac{1}{2}u_{0}^{2}}dx.$$
So we assume by contradiction that this is not the case, i.e. 
$$\|u_{0}\|_{H}^{2}>\lambda \int_{I}u_{0}^{2}e^{\frac{1}{2}u_{0}^{2}}dx.$$
Then from Lemma \ref{comp}, Lemma \ref{lemmaaJ2} and the weak convergence, we have that

$$\|u_{0}\|_{H}^{2}\leq \liminf_{k\to\infty} \|u_{k}\|_{H}^{2}=\liminf_{k\to\infty} \lambda \int_{I}u_{k}^{2}e^{\frac{1}{2}u_{k}^{2}}dx =\lambda \int_{I}u_{0}^{2}e^{\frac{1}{2}u_{0}^{2}}dx,$$
again leading to a contradiction.\\
From Lemma \ref{nihari}, we have that there exists $0<t\leq 1$ such that $tu_{0}\in N(J)$. Taking Remark \ref{rmkI} into account we get
$$a(J)\leq J(tu_{0})=Q(tu_{0})\leq Q(u_{0})\leq \liminf_{k\to\infty} Q(u_{k}) =a(J).$$
It follows that $t=1$, since otherwise the second inequality above would be strict. Then $u_{0}\in N(J)$ and $J(u_{0})=a(J)$. By Lemma \ref{lemmacrit} we have $J'(u_0)=0$
\hfill$\square$


\section{Proof of Theorem \ref{MT3}}

For $u\in H^{\frac12,2}(\mathbb{R})$ we set $|u|^{*}:\R{}\to \R{}_+$ to be its non-increasing symmetric rearrangement, whose definition we shall now recall.
For a measurable set $A\subset \mathbb{R}$, we define 
$$A^{*}=(-|A|/2,|A|/2).$$
The set $A^{*}$ is symmetric (with respect to $0$) and  $|A^*|=|A|$. For a non-negative measurable function $f$, such that 
$$|\{x\in \mathbb{R} : f(x)>t\}|<\infty \quad \text{ for every } t>0,$$
we define the symmetric non-increasing rearrangement of $f$ by
$$f^{*}(x)=\int_{0}^{\infty} \chi_{\{y\in\R{} : f(y)>t\}^{*}}(x) dt.$$
Notice that $f^*$ is even, i.e. $f^*(x)=f^*(-x)$ and non-increasing (on $[0,\infty)$).

We will state here the two properties that we shall use in the proof of Theorem \ref{MT3}. The following one is proven e.g. in \cite[Section 3.3]{LL}.

\begin{prop}\label{propu*}
Given a measurable function $F:\mathbb{R}\to \mathbb{R}$ and a non-negative non-decreasing function $f:\R{}\to\R{}$ it holds
$$\int_{\mathbb{R}}F(f)dx =\int_{\mathbb{R}}F(f^{*})dx.$$
\end{prop}

The following P\'olya-Szeg\H{o} type inequality can be found e.g. in \cite{jam} (Inequality (3.6)) or \cite{Park}.

\begin{trm}\label{pol}
Let $u\in H^{s,2}(\mathbb{R})$ for $0<s<1$. Then
$$\int_{\mathbb{R}}|(-\Delta)^{s}|u|^{*}|^{2}dx\leq \int_{\mathbb{R}}|(-\Delta)^{s}u|^{2}dx.$$ 
\end{trm}

Now given $u\in H^{\frac12,2}(\R{})$, from Proposition \ref{propu*} we get
$$\int_{\mathbb{R}}\left( e^{\pi u^{2}}-1 \right) dx =\int_{\mathbb{R}}\left( e^{\pi (|u|^*)^{2}}-1 \right)dx,\quad \| |u|^*\|_{L^2}=\|u\|_{L^2},$$
and according to Theorem \ref{pol}
$$\||u|^*\|_{H^{\frac12,2}(\R{})}^2=\||u|^*\|_{L^2(\R{})}^2+\int_{\mathbb{R}}|(-\Delta)^{\frac{1}{4}}|u|^{*}|^{2}dx\leq \|u\|_{L^2(\R{})}^2+\int_{\mathbb{R}}|(-\Delta)^{\frac{1}{4}}u|^{2}dx = \|u\|_{H^{\frac12,2}(\R{})}^2.$$
Therefore in the rest of the proof of \eqref{stimaMT3} we may assume that $u\in H^{\frac12,2}(\R{})$ is even, non-increasing on $[0,\infty)$, and $\|u\|_{H^{\frac12,2}(\R{})}\le 1$.

We write
$$\int_{\mathbb{R}}\left(e^{\pi u^{2}}-1 \right)dx =\int_{\R{}\setminus I}\left(e^{\pi u^{2}}-1 \right)dx+\int_{I}\left(e^{\pi u^{2}}-1\right) dx=:(I)+(II),$$
where $I=(-1/2,1/2)$.
We start by bounding $(I)$.  By monotone convergence
$$(I)=\sum_{k=1}^{\infty} \int_{I^c} \pi^{k} \frac{u^{2k}}{k!}dx.$$
Since $u$ is even and non-increasing, for $x\ne 0$ we have
\begin{equation}\label{radstima}
u^2(x)\le \frac{1}{2|x|}\int_{-|x|}^{|x|} u^2(y)dy\le \frac{\|u\|_{L^2}^2}{2|x|},
\end{equation}
hence for $k\ge 2$ we bound
$$\int_{I^c}u^{2k}dx\leq 2^{1-k} \|u\|_{L^{2}(\R{})}^{2k}\int_\frac12^{\infty}\frac{1}{x^{k}} dx= \frac{ \|u\|_{L^{2}(\R{})}^{2k}}{(k-1)}.$$
It follows that
$$\sum_{k=2}^{\infty} \int_{I^c} \pi^{k} \frac{u^{2k}}{k!}dx \leq \sum_{k=2}^{\infty} \frac{(\pi \|u\|_{L^{2}}^{2})^{k}}{k! (k-1)}. $$
Thus, since $\|u\|_{L^2(\R{})}\le 1$ we estimate
$$(I)\leq \pi \|u\|_{L^{2}(\R{})}^{2}\left(1+\sum_{k=1}^{\infty}\frac{\left(\pi \|u\|^{2}_{L^2(\R{})}\right)^k}{(k+1)!k} \right) \leq C.$$
We shall now bound $(II)$. We define the function $v:\R{}\to \R{}$ as follows 
$$v(x)=\left\{ \begin{array}{ll}
u(x)-u(\tfrac12) &\text{if } |x|\leq \frac12\\
0 &\text{if }|x|>\frac12\rule{0cm}{.5cm}.
\end{array}
\right. 
$$
Then with \eqref{radstima} and the estimate $2a\leq a^{2}+1$, we find
\begin{equation}\label{equv}
\begin{split}
u^{2}&\le v^{2}+2vu(\tfrac12)+u(\tfrac{1}{2})^{2}\\
&\le v^{2}+2v\|u\|_{L^2(\R{})}+\|u\|_{L^2(\R{})}^{2}\\
&\leq v^{2}+v^{2}\|u\|^{2}_{L^{2}(\R{})}+1 +\|u\|^{2}_{L^{2}(\R{})}\\
&\leq v^{2}\left(1+\|u\|^{2}_{L^{2}(\R{})}\right)+2.
\end{split}
\end{equation}
Now, recalling that $u$ is decreasing we have for $x\in I =[-\tfrac12,\tfrac12]$
\[\begin{split}
\int_{\mathbb{R}}\frac{(v(x)-v(y))^{2}}{(x-y)^{2}}dy&=\int_{I}\frac{(u(x)-u(y))^{2}}{(x-y)^{2}}dy+\int_{I^c}\frac{(u(x)-u(\tfrac12))^{2}}{(x-y)^{2}}dy\\
&\le\int_{\mathbb{R}}\frac{(u(x)-u(y))^{2}}{(x-y)^{2}}dy.
\end{split}\]
Notice that the last integral converges for a.e. $x\in I$ thanks to Proposition \ref{HW} and Fubini's theorem. Similarly for $x\in I^c$
\[\begin{split}
\int_{\mathbb{R}}\frac{(v(x)-v(y))^{2}}{(x-y)^{2}}dy&=\int_{I}\frac{(u(\tfrac12)-u(y))^{2}}{(x-y)^{2}}dy\\
&\leq \int_{I}\frac{(u(x)-u(y))^{2}}{(x-y)^{2}}dy\\
&\leq \int_{\mathbb{R}}\frac{(u(x)-u(y))^{2}}{(x-y)^{2}}dy.
\end{split}\]
Integrating with respect to $x$  we obtain
\[\begin{split}
\|(-\Delta)^\frac14 v\|^2_{L^2(\R{})} &= \frac{1}{C_s^2} \int_{\mathbb{R}}\int_{\mathbb{R}}\frac{(v(x)-v(y))^{2}}{(x-y)^{2}}dydx\\
& \leq \frac{1}{C_s^2} \int_{\mathbb{R}}\int_{\mathbb{R}}\frac{(u(x)-u(y))^{2}}{(x-y)^{2}}dydx\\
&= \|(-\Delta)^\frac14 u\|^2_{L^2(\R{})},
\end{split}
\]
where $C_s$ is as in Proposition \ref{HW} below.
Thus, since $\|u\|_H\leq 1$, 
$$\|(-\Delta)^\frac14 v\|^2_{L^2(\R{})}\leq \|(-\Delta)^\frac14 u\|^2_{L^2(\R{})}\leq 1-\|u\|^{2}_{L^{2}(\R{})}.$$
Therefore, if we set $w=v \sqrt{1+\|u\|^{2}_{L^2(\R{}) }}$, we have
$$\|(-\Delta)^\frac14 w\|^2_{L^2(\R{})}\leq \left(1+\|u\|^{2}_{L^2(\R{})} \right)  \left(1-\|u\|^{2}_{L^2(\R{})}\right)\le 1,$$
hence, using the Moser-Trudinger inequality on the interval $I=(-1/2,1/2)$ (Theorem \ref{MT2}), one has
$$\int_I e^{\pi w^{2}}dx<C,$$
and using \eqref{equv}
$$\int_I e^{\pi u^{2}}dx \leq e^{2\pi}\int_I e^{\pi w^{2}}dx \leq C,$$
which completes the proof of \eqref{stimaMT3}.

\medskip

It remains to prove \eqref{stimaMT3s}. Given $\tau>2$ consider the function
$$f=f_\tau:=\frac{1}{2\tau \sqrt{|x|}} \chi_{\{x\in\R{}:{r<|x|<\delta}\}},\quad \delta:=\frac{1}{\tau}, \quad r:=\frac{1}{\tau e^\tau}.$$
Notice that $\|f\|_{L^2(\R{})}^2=(2\tau)^{-1}$.
Fix a smooth even function $\psi:\R{}\to [0,1]$ with $\psi\equiv 1$ in $[-\frac{1}{2}, \frac{1}{2}]$ and $\mathrm{supp}(\psi)\subset(-1,1)$. For $x\in \R{}$ we set
$$u(x)=\psi (x)(F_\frac14 *f)(x),$$
where $F_\frac14(x)=(2\pi|x|)^{-\frac12}$ is as in Lemma \ref{green2}.
Clearly $u\equiv 0$ in $\R{}\setminus I$, and $u$ is non-negative and even everywhere.

\medskip

In the rest of the proof $s=\frac14$. Notice that $(-\Delta)^s (F_s*f)=f$. This follows easily from Lemma \ref{green2} and the properties of the Fourier transform, see e.g. \cite[Corollary 5.10]{LL}. Then we compute

\begin{equation}\label{test}
(-\Delta)^su=f+(-\Delta)^s[(\psi-1)(F_s*f)]=:f+v,
\end{equation}
and set $g(x,y)=(\psi-1)(x)F_s(x-y)$. Notice that $g$ is smooth in $\R{}\times (-\frac12,\frac12)$. We write
\[\begin{split}
v(x)&=(-\Delta)^s \int_{\R{}} g(x,y)f(y)dy\\
&=\int_{\{r<|y|<\delta\}}(-\Delta_x)^sg(x,y) f(y)dy,
\end{split}\]
where we used Proposition \ref{lapint} and Fubini's theorem.
With Jensen's inequality
\begin{equation}\label{estv}
\begin{split}
\|v\|_{L^{2}(\R{})}^2 &=\int_{\R{}}\left|\int_{\{r<|y|<\delta\}}(-\Delta_x)^sg(x,y) f(y)dy\right|^{2}dx\\
&\leq 2(\delta-r)\int_{\{r<|y|<\delta\}} f(y)^{2}\int_{\R{}}\left|(-\Delta_x)^s g(x,y)\right|^{2}dx dy\\ 
&\leq  2\delta\|f\|^{2}_{L^{2}(\R{})}\sup_{|y|\in [r,\delta]} \int_{\R{}}\left| (-\Delta_x)^sg(x,y)\right|^{2}dx\\
&\le C (\delta \tau^{-1})=O(\tau^{-2}),
\end{split}
\end{equation}
where we used that
$$\sup_{|y|\in [r,\delta]} \int_{\R{}}\left| (-\Delta_x)^sg(x,y)\right|^{2}dx<\infty.$$
This in turn can be seen noticing that $(-\Delta_x)^s g(x,y)$ is smooth, hence bounded on $[-R,R]\times [r,\delta]$ for every $R$, and for $|x|$ large and $r\le |y|\le \delta$, using Proposition \ref{lapint}
\[\begin{split}
(-\Delta_x)^s g(x,y)&= C_s\int_{\R{}}\frac{-F_s(x-y)-(\psi(z)-1) F_s(z-y)}{|z-x|^{1+2s}}dz \\
&=C_s \int_{-1}^1\frac{-\psi(z) F_s(z-y)}{|z-x|^{1+2s}}dz -(-\Delta)^s F_s(x-y)\\
&=O(|x|^{-1-2s})\quad \text{uniformly for }|y|\le \frac{1}{2},
\end{split}\]
where we also used that $(-\Delta)^s F_s =0$ away from the origin, see Lemma \ref{green2}.
Actually, with the same estimates we get
\[\begin{split}
\int_{-\delta}^\delta |v|^2dx  &\leq 2 (\delta-r)\|f\|^{2}_{L^{2}(\R{})}\int_{-\delta}^\delta \sup_{(x,y)\in [-\delta,\delta]^2} \left| (-\Delta_x)^sg(x,y)\right|^{2}dx\\
&\le C\delta^2\|f\|_{L^2(\R{})}^2=O(\tau^{-3}).
\end{split}\]
Therefore, using H\"older's inequality and that $\mathrm{supp}(f)\subset [-\delta,\delta]$ we get 
\begin{equation}\label{stimadeltau}
\|(-\Delta)^s u\|_{L^{2}(\R{})}^{2} = \|f\|_{L^2}^2+\|v\|_{L^2}^2+2\int_{-\delta}^\delta fvdx=\frac{1}{2\tau}+O(\tau^{-2}),\quad \text{as }\tau\to\infty.
\end{equation}

\medskip

\noindent We now estimate $u$.  For $0 < x < r$, with the change of variable $\tilde y=\sqrt{\frac{y}{x}}$ we have

\begin{equation*}
\begin{split}
u(x)&= \frac{1}{2\tau \sqrt{2\pi}}\int_r^\delta \left(\frac{1}{\sqrt{(y-x) y}}+\frac{1}{\sqrt{(y+x) y}}\right)dy\\
&=\frac{1}{\tau\sqrt{2\pi}}\int_{\sqrt\frac{r}{x}}^{\sqrt\frac{\delta}{x}}\left(\frac{1}{\sqrt{\tilde y^2-1}}+\frac{1}{\sqrt{\tilde y^2+1}} \right)d\tilde y\\
&=\frac{1}{\tau\sqrt{2\pi}}\left(\log(\sqrt{\tilde y^2-1}+\tilde y)\bigg|_{\sqrt\frac{r}{x}}^{\sqrt\frac{\delta}{x}}+\log(\sqrt{\tilde y^2+1}+\tilde y)\bigg|_{\sqrt\frac{r}{x}}^{\sqrt\frac{\delta}{x}}\right)\\
&=\frac{1}{\sqrt{2\pi}}+O(\tau^{-1}),
\end{split}
\end{equation*}
with $|\tau O(\tau^{-1})|\le C$ as $\tau\to\infty$ with $C$ independent of $x\in [0,r]$.

Similarly for $r<x<\delta$ we write
\begin{equation*}
\begin{split}
u(x)&\le  \frac{1}{\tau\sqrt{2\pi}}\left[ \int_r^x \frac{dy}{\sqrt{(x-y)y}} + \int_x^\delta \frac{dy}{\sqrt{(x-y) y}}\right]\\
&=\frac{2}{\tau\sqrt{2\pi}}\left[ \int^{1}_{\sqrt{\frac{r}{x}}}\frac{d\tilde y}{\sqrt{1-\tilde y^{2}}} +\log( \sqrt{\tilde y^2-1}+\tilde y)\Big|^{\sqrt\frac{\delta}{x}}_{1} \right]\\
&=\frac{1}{\tau \sqrt{2\pi}}\left[\log\left( \frac{\delta}{x}\right)+O(1) \right],
\end{split}
\end{equation*}
since $\int_0^1 \frac{d\tilde y}{\sqrt{1-\tilde y^{2}}}< \infty$. Here $|O(1)|\le C$ as $\tau \to \infty$ with $C$ independent of $x\in (r,\delta)$.

\noindent When $\delta <x<1$ similar to the previous computation, and recalling that $0\le \psi\le 1$,
\begin{equation*}
\begin{split}
u(x)&\leq \frac{1}{\tau \sqrt{2\pi}} \int_r^\delta \frac{dy}{\sqrt{(x-y) y}} =\frac{2}{\tau\sqrt{2\pi}}\int_{\sqrt{\frac{r}{x}}}^{\sqrt{\frac{\delta}{x}}}\frac{d\tilde y}{\sqrt{1-\tilde y^{2}}}\le \frac{2}{\tau\sqrt{2\pi}}\int_0^1\frac{d\tilde y}{\sqrt{1-\tilde y^{2}}}=O(\tau^{-1}),
\end{split}
\end{equation*}
with $|\tau O(\tau^{-1})|\le C$ as $\tau \to \infty$ with $C$ independent of $x\in (0,1)$.
Thus
\begin{equation}\label{stimeu}
\left\{
\begin{array}{ll}
u(x)= \frac{1}{\sqrt{2\pi}}+O(\tau^{-1}) & \text{for } 0<x<r\\
u(x)\le \frac{2}{\tau\sqrt{2\pi}}\log\left( \frac{\delta}{x}\right)+O(\tau^{-1}) & \text{for } r<x<\delta \\
u(x) =O(\tau^{-1}) & \text{for } \delta<x<1.
\end{array}
\right.
\end{equation}
Of course the same bounds hold for $x<0$ since $u$ is even.

We now want to estimate $\|u\|_{L^{2}(\R{})}^{2}$. We have 
$$\int_{0}^{r}u^{2} dx =r\left(\frac{1}{2\pi} + O(\tau^{-1})\right)=O(\tau^{-2}).$$
For $x\in [r,\delta]$ we have from \eqref{stimeu}
$$u(x)^2\le \frac{C}{\tau^2}\left(\log^2\left(\frac\delta x\right)+\log\left(\frac\delta x\right) +1\right)\le  \frac{2C}{\tau^2}\left(\log^2\left(\frac\delta x\right)+1\right) .$$
Then, since
$$\int_{r}^{\delta}\log^{2}\left(\frac{\delta}{x}\right)dx = x\left(\log^{2}\left(\frac{\delta}{x}\right)+2\log\left(\frac{\delta}{x}\right)+2\right)\bigg|_{r}^{\delta}\leq 2\delta =O(\tau^{-1}),$$
we bound
$$\int_{r}^{\delta}u^{2}dx =O(\tau^{-3}).$$
Finally, still using \eqref{stimeu},
$$\int_{\delta}^{1}u^{2}dx =O(\tau^{-2}).$$
Also considering \eqref{stimadeltau}, we conclude
\begin{equation}\label{stimaufin}
\|u\|_{L^{2}(\R{})}^{2}=2\|u\|_{L^2([0,1])}^2= O(\tau^{-2}),\quad \|u\|_{H^{\frac{1}{2},2}(\R{})}^2=\frac{1}{2\tau}+O(\tau^{-2}).
\end{equation}
Setting $w_\tau:=u \|u\|_{H^{\frac{1}{2},2}(\R{})}^{-1}$, and using \eqref{stimeu} and \eqref{stimaufin}, we conclude
\[
\int_{-r}^{r}|w_\tau|^2 \left(e^{\pi w_\tau^2}-1\right)dx \ge\int_{-r}^r\left(\frac{\tau+O(1)}{\pi}\right)\left( e^{\tau+O(1)} -1\right) dx\\
\ge \frac{r \tau e^{\tau}}{C} =\frac{1}{C},
\]
therefore
$$\lim_{\tau\to\infty}\int_{\R{}}h(w_\tau) \left(e^{\pi w_\tau^2}-1\right)dx \ge \int_{-r}^{r} h(w_\tau) \left(e^{\pi w_\tau^2}-1\right)dx\to\infty $$
as $\tau\to \infty$, for any $h$ satisfying \eqref{condh2}.
\hfill$\square$

\appendix

\section{Some useful results}
We define
\begin{equation}\label{defWsp}
W^{s,p}(\R{}):=\left\{u\in L^p(\R{}): [u]_{W^{s,p}(\R{})}^p:=\int_{\R{}}\int_{\R{}}\frac{|u(x)-u(y)|^p}{|x-y|^{1+sp}}dxdy<\infty\right\}.
\end{equation}

\begin{prop}\label{HW}
For $s\in (0,1)$ we have, $[u]_{W^{s,2}(\R{})}<\infty$ if and only if $(-\Delta)^\frac{s}{2}u\in L^2(\R{})$, and in this case
$$[u]_{W^{s,2}(\R{})}= C_s \|(-\Delta)^\frac{s}{2}u\|_{L^2(\R{})},$$
where $[u]_{W^{s,2}(\R{})}$ is as in \eqref{defWsp} and $C_s$ depends only on $s$. In particular $H^{s,2}(\R{})= W^{s,2}(\R{})$.
\end{prop}

\begin{proof}
See e.g. Proposition 3.6 in \cite{DNPV}.
\end{proof}

\medskip

Define the bilinear form
$$\M{B}_s(u,v)=\int_{\R{}}\int_{\R{}} \frac{(u(x)-u(y))(v(x)-v(y))}{|x-y|^{1+2s}}dxdy,\quad \text{for }u,v\in H^{s,2}(\R{}),$$
where the double integral is well defined thanks to H\"older's inequality and Proposition \ref{HW}. 

The following simple and well-known existence result proves useful. A proof can be found (in a more general setting) in \cite{FKV}.

\begin{trm}\label{trmexist}
Given $s\in (0,1)$, $f\in L^2(I)$ and $g:\R{}\to \R{}$ such that
\begin{equation}\label{intg}
\int_I\int_{\R{}}\frac{(g(x)-g(y))^2}{|x-y|^{1+2s}}dxdy<\infty,
\end{equation}
there exists a unique function $u\in \tilde H^{s,2}(I)+g$ solving the problem
\begin{equation}\label{uvar}
\M{B}_s(u,v)=\int_{\R{}} f v dx\quad \text{for every }v\in \tilde H^{s,2}(I).
\end{equation}
Moreover such $u$ satisfies $(-\Delta)^s u=\frac{C_s}{2}f$ in $I$ in the sense of distributions, i.e.
\begin{equation}\label{uweak}
\int_{\R{}} u(-\Delta)^s \varphi dx=\frac{C_s}{2}\int_{\R{}} f\varphi dx\quad \text{for every }\varphi \in C^\infty_c(I),
\end{equation}
where $C_s$ is the constant in Proposition \ref{lapint}.
\end{trm}

The following version of the maximum principle is a special case of Theorem 4.1 in \cite{FKV}. 

\begin{prop}\label{maxprinc2} Let $u\in \tilde H^{s,2}(I)+g$ solve \eqref{uvar} for some $f\in L^2(I)$ with $f\ge 0$ and $g$ satisfying \eqref{intg} and $g\ge 0$ in $I^c$. Then $u\ge 0$.
\end{prop}

\begin{proof} From Proposition \ref{HW} it easily follows $u^-:= \min\left\{u,0\right\} \in \tilde H^{s,2}(I)$. Then according to \eqref{uvar} we have
\[\begin{split}
0\ge \M{B}_s(u,u^-)&=\int_{\R{}}\int_{\R{}} \frac{(u^+(x)+u^-(x)-u^+(y)-u^-(y))(u^-(x)-u^-(y))}{|x-y|^{1+2s}}dxdy\\
&=\int_{\R{}}\int_{\R{}} \frac{(u^-(x)-u^-(y))^2}{|x-y|^{1+2s}}dxdy-2\int_{\R{}}\int_{\R{}}\frac{u^+(x)u^-(y)}{|x-y|^{1+2s}}dxdy
\end{split}\]
where we used that $u^+u^-=0$. Since the second term in the last equality is non-negative, it follows at once that $u^-\equiv 0$, hence $u\ge 0$.
\end{proof}

\begin{prop}\label{trmbordo}
Let $u\in \tilde H^{s,2}(I)$ be as in Theorem \ref{trmexist} (with $g=0$), where we further assume $f\in L^\infty(I)$.  Then
$$|u(x)|\le C\|f\|_{L^\infty(I)}(\mathrm{dist}(x,\de I))^s$$
for every $x\in I$. In particular $u$ is bounded in $I$ and continuous at $\de I$.
\end{prop}

\begin{proof} This proof is inspired from \cite{XR}, where a much stronger result is proven, i.e. $u/(\mathrm{dist}(\cdot, \de I))^s\in C^\alpha (\bar I)$ for some $\alpha>0$.

To prove the proposition we assume as usual that $I=(-1,1)$ and recall that
\[
w(x):=\left\{
\begin{array}{ll}
(1-|x|^2)^s &\text{for }x\in (-1,1)\\
0 &\text{for }|x|\ge 1
\end{array}\right.
\]
belongs to $ \tilde H^{s,2}(I)$ and solves $(-\Delta)^s w= \gamma_s$ for a positive constant $\gamma_s$, in the sense of Theorem \ref{trmexist}, i.e. \eqref{uvar} holds with $u=w$ and $f\equiv\gamma_s$ (see e.g. \cite{get}). Then
$$-\frac{(-\Delta)^s w}{\gamma_s}\le \frac{(-\Delta)^s u}{\|f\|_{L^\infty(I)}}\le \frac{(-\Delta)^s w}{\gamma_s}$$
and Proposition \ref{maxprinc2} gives at once
$$-\frac{\|f\|_{L^\infty(I)}}{\gamma_s} w\le u\le  \frac{\|f\|_{L^\infty(I)}}{\gamma_s}w \quad \text{in }I.$$
We conclude noticing that $0\le w(x)\le 2^s (\mathrm{dist}(x,\de I))^s$.
\end{proof}

The following density result is known for an arbitrary domain in $\R{n}$. On the other hand, its proof is quite complex in such a generality, hence we provide a short elementary proof which fits the case of an interval.

\begin{lemma}\label{lemmadens} 
For $s\in (0,1)$ and $p\in [1,\infty)$ the sets $C^\infty_c(I)$ ($I\Subset\R{}$ is a bounded interval) is dense in $\tilde H^{s,p}(I)$.
\end{lemma}

\begin{proof}
 Without loss of generality we consider $I=(-1,1)$. Given $u\in \tilde H^{s,p}(I)$ and $\lambda>1$, set $u_\lambda(x):=u(\lambda x)$. We claim that $u_\lambda\to u$ in $\tilde H^{s,p}(I)$ as $\lambda\to 1$. Indeed
\[\begin{split}
\|u_\lambda-u\|_{H^{s,p}(\R{})}^p&
 =\|u-u_\lambda\|_{L^p(\R{})}^p+ \|\lambda^s f_\lambda -f\|_{L^p(\R{})}^p,
\end{split}\]
where $f=(-\Delta)^\frac s2 u$ and $f_\lambda(x):=f(\lambda x)$. Since $f\in L^p(\R{})$ it follows that $\|\lambda^s f_\lambda -f\|_{L^p(\R{})}\to 0$ as $\lambda\to 1$, since this is obviously  true for $f\in C^0(\R{})$ with compact support, and for a general $f\in L^p(\R{})$ it can be proven by approximation in the following standard way. Given $\ve>0$ choose $f_\ve\in C^0(\R{})$ with compact support and $\|f_\ve-f\|_{L^p(\R{})}\le \ve$. Then by the Minkowski inequality
\[
\begin{split}
\|\lambda^s f_\lambda-f\|_{L^p(\R{})}&\le \|\lambda^s f_\lambda-\lambda^s f_{\ve,\lambda}\|_{L^p(\R{})}+ \|\lambda^s f_{\ve,\lambda}-f_\ve \|_{L^p(\R{})} + \| f_\ve- f\|_{L^p(\R{})}\\
&\le \ve \lambda^{s-\frac1p} +\|\lambda^s f_{\ve,\lambda}-f_\ve \|_{L^p(\R{})} +\ve,
\end{split}
\]
and it suffices to let $\lambda\to 1$ and $\ve \to 0$. Similarly $\|u-u_\lambda\|_{L^p(\R{})}^p\to 0$ as $\lambda\to 1$.

Now given $\delta >0$ fix $\lambda>1$ such that $\|u_\lambda -u\|_{H^{s,p}(\R{})}<\delta$ and let $\rho$ be a mollifying kernel, i.e. a smooth non-negative function supported in $I$ with $\int_I \rho dx=1$. Also set $\rho_\ve(x):=\ve^{-1}\rho(\ve^{-1}x)$. Then noticing that $u_\lambda$ is supported in $[-\lambda^{-1},\lambda^{-1}]\Subset I$, for $\ve>0$ sufficiently small we have that $\rho_\ve * u_\lambda\in C^\infty_c(I)$. To conclude the proof notice that
$$\rho_\ve * u_\lambda \to u_\lambda \text{ in }\tilde H^{s,p}(I) \text{ as }\ve\to 0,$$
since
$$(-\Delta)^\frac s2 (\rho_\ve * u_\lambda)= \rho_\ve * (-\Delta)^\frac s2 u_\lambda \to (-\Delta)^\frac s2 u_\lambda \quad \text{in } L^p(\R{})\text{ as }\ve\to 0,$$
and use the Minkowski inequality to conclude that $\rho_\ve * u_\lambda \to u$ in $\tilde H^{s,p}(I)$ as $\ve\to 0$ and $\lambda\downarrow 1$.
\end{proof}

\begin{prop}\label{maxprinc} Let $I\Subset \R{}$ be a bounded interval and $s\in (0,1)$. Let $u\in L_s(\R{})$ satisfy $(-\Delta)^s u\ge 0$ in $I$ (i.e. $\langle u,(-\Delta)^s\varphi\rangle\ge 0$ for every $\varphi\in C^\infty_c(I)$ with $\varphi\ge 0$),  $u\ge 0$ in $I^c$ and
\begin{equation}\label{usemic}
\liminf_{x\to \de I}u(x)\ge 0.
\end{equation}
Then $u\ge 0$ in $I$. More precisely, either $u>0$ in $I$, or $u\equiv 0$ in $\R{}$.
\end{prop}

\begin{proof} This is a special case of Proposition 2.17 in \cite{Sil}.
\end{proof}

\begin{rmk} The statement of Proposition 2.17 in \cite{Sil} is slightly different, since it assumes $u$ to be lower-semicontinuous in $\bar I$. On the other hand, lower semicontinuity inside $I$ already follows from \cite[Prop. 2.15]{Sil}. What really matters is condition \eqref{usemic}. That an assumption of this kind (possibly weaker) is needed follows for instance from the example of Lemma 3.2.4 in \cite{Aba}.
\end{rmk}

The following way of computing the fractional Laplacian of a sufficiently regular function is often used.

\begin{prop}\label{lapint} For an open interval $J\subset\R{}$, let $s\in (0,\frac{1}{2})$ and $u\in L_s(\R{})\cap C^{0,\alpha}(J)$ for some $\alpha\in  (2s,1]$, or $s\in [\frac{1}{2},1)$ and $u\in L_s(\R{})\cap C^{1,\alpha}(J)$ for some $\alpha\in  (2s-1,1]$ . Then $((-\Delta)^s u)|_J\in C^0(J)$ and
$$(-\Delta)^su(x)=C_s P.V.\int_\R{} \frac{u(x)-u(y)}{|x-y|^{1+2s}}dy:= C_s\lim_{\ve\to 0}\int_{\R{}\setminus [x-\ve,x+\ve]}\frac{u(x)-u(y)}{|x-y|^{1+2s}}dy$$
for every $x\in J$. This means that 
$$\langle (-\Delta)^su, \varphi\rangle =C_s\int_{\R{}} \varphi(x)\,P.V.\int_\R{} \frac{u(x)-u(y)}{|x-y|^{1+2s}}dy\,dx,\quad \text{for every } \varphi\in C^\infty_c(J).$$
\end{prop}

\begin{proof} See e.g. \cite[Prop. 2.4]{Sil}
\end{proof}

\begin{lemma}\label{phi1} Let $\varphi_1\in H=\tilde H^{\frac{1}{2},2}(I)$ be an eigenfunction corresponding to the first eigenvalue $\lambda_1(I)$ of $(-\Delta)^\frac{1}{2}$ on $I$. Then $\varphi_1>0$ a.e. on $I$ or $\varphi_1<0$ a.e. on $I$ and the corresponding eigenspace has dimension $1$.
\end{lemma}
\begin{proof}
Recall that the first eigenvalue $\lambda_{1}(I)$ can be characterised by minimizing the following functional $$F(u)=\frac{\|u\|_{H}^{2}}{\int_{I} u^{2} dx},$$
that is,
$$\lambda_{1}(I)=\min_{u\in H\setminus \{0\} } F(u).$$
On the other hand using Proposition \ref{HW} we get that for any $u\in H$
\begin{equation}\label{|u|}
\|u\|_{H}^{2}=\int_{\mathbb{R}}\int_{\mathbb{R}} \frac{(u(x)-u(y))^{2}}{(x-y)^{2}}dxdy\geq \int_{\mathbb{R}}\int_{\mathbb{R}} \frac{(|u(x)|-|u(y)|)^{2}}{(x-y)^{2}}dxdy=\| |u| \|_H^{2},
\end{equation}
hence, $F(|u|)\leq F(u)$, and $F(u)=F(|u|)$ if and only if $u$ is non-negative or non-positive. Therefore if $F(\varphi_{1})=\lambda_{1}$, then $\varphi_1$ does not change sign. Moreover Theorem $A.1$ in \cite{BF} gives us $\varphi_1 >0$ or $\varphi_1<0$ almost everywhere in $I$. Any other eigenfunction corresponding to $\lambda_1$ must also have fixed sign, hence it cannot be orthogonal to $\varphi_1$, therefore it is a multiple of $\varphi_1$.
\end{proof}

\begin{lemma}\label{lemmafkL} Consider a sequence $(f_k)\subset L^1(I)$ with $f_k\to f$ a.e. and with
\begin{equation}\label{fkL}
\int_{\{f_k>L\}} f_k dx =o(1),
\end{equation}
with $o(1)\to 0$ as $L\to\infty$ uniformly with respect to $k$. Then $f_k\to f$ in $L^1(I)$.
\end{lemma}

\begin{proof}
From the dominated convergence theorem
$$\min\{f_k, L\}\to \min\{f, L\}\quad \text{in }L^1(I),$$
and the convergence of $f_k$ to $f$ in $L^1$ follows at once from \eqref{fkL} and the triangle inequality.
\end{proof}

\paragraph{Acknowledgements}

We would like to thank the anonymous referees for the very careful reading and for the very useful suggestions.


\begin{thebibliography}{99}


\bibitem{Aba} \textsc{N. Abatangelo}, \emph{Large $s$-harmonic functions and boundary blow-up solutions for the fractional Laplacian}, Discrete Contin. Dynam. Systems A \textbf{35} (2015), 5555-5607.

\bibitem{ada} \textsc{D. Adams,} \emph{A sharp inequality of J. Moser for higher order derivatives}, Ann. of Math. \textbf{128} (1988), 385-398.

\bibitem{A}\textsc{Adimurthi,}  \emph{Existence of positive solutions of the semilinear Dirichlet problem with critical growth for the n-Laplacian}, Ann. Scuola Norm. Sup. Pisa Cl. Sci. \textbf{(4)} 17 (1990), no. 3, 393-413.


\bibitem{AS} \textsc{Adimurthi, M. Struwe,} \emph{Global compactness properties of semilinear elliptic equations with critical exponential growth}, J. Funct. Anal. \textbf{175} (2000), 125-167.

\bibitem{AP} \textsc{F. V. Atkinson, L. A. Peletier,} \emph{Ground states and Dirichlet problems for $-\Delta u = f(u)$ in $\R{2}$}, Arch. Rational Mech. Anal. \textbf{21} (1986), 147-165.



\bibitem{BGR} \textsc{R. M. Blumenthal, R. Getoor, D. B. Ray,} \emph{On the distribution of first hits for the symmetric stable processes,} Trans. Amer. Math. Soc. \textbf{99} (1961), 540-554.

\bibitem{BLW} \textsc{M. Birkner, J. A. Lopez-Mimbela, A. Wakolbinger,} \emph{Comparison results and steady states for the Fujita equation with fractional Laplacian}, Ann. I. H. Poincar\'e \textbf{22} (2005) 83-97.

\bibitem{bon} \textsc{J. M. Bony,} \emph{Cours d'Analyse. Th\'eorie des distributions et analyse de Fourier}, Ecole Polytechnique, (2001). ISBN: 2-7302-0775-9.






\bibitem{BF}\textsc{L. Brasco, G. Franzina,}\emph{ Convexity properties of Dirichlet integrals and Picone-type inequalities}, Kodai Math. J. \textbf{37} (2014), 769-799.

\bibitem{CB}\textsc{C. Bucur,}\emph{ Some observations on the Green function for the ball in the fractional Laplace framework}, Preprint (2015).


\bibitem{DLMR} \textsc{F. Da Lio, L. Martinazzi, T. Rivi\`ere}, \emph{Non-local Liouville equation}, Analysis \& PDE \textbf{8}, No. 7 (2015), 1757-1805.


\bibitem{DNPV}\textsc{E. Di Nezza, G. Palatucci, E. Valdinoci}, \emph{Hitchhiker's guide to the fractional Sobolev spaces}, Bull. Sci. Math., Vol. \textbf{136} No. 5 (2012), 521-573.

\bibitem{dru} \textsc{O. Druet}, \emph{Multibumps analysis in dimension $2$: quantification of blow-up levels}, Duke Math. J. \textbf{132} (2006), 217-269.

\bibitem{FKV} \textsc{M. Felsinger, M. Kassmann, P. Voigt}, \emph{The Dirichlet problem for nonlocal operators}, Math. Z. \textbf{279} (2015), no.3-4, 779-809.


\bibitem{get} \textsc{R. K. Getoor}, \emph{First passage times for symmetric stable processes in spaces}, Trans. Amer. Math. Soc \textbf{101} (1961), 75-90.


\bibitem{Gr0} \textsc{G. Grubb}, \emph{Fractional Laplacians on domains, a development of H\"ormander's theory of mu-transmission pseudodifferential operators}, Adv. Math. \textbf{268} (2015), 478-528. 


\bibitem{IS} \textsc{A. Iannizzotto, M. Squassina}, \emph{1/2-Laplacian problems with exponential nonlinearity},
J. Math. Anal. Appl. \textbf{414} (2014), 372-385.

\bibitem{jam} \textsc{P. Jaming}, \emph{On the Fourier transform of the symmetric decreasing rearrangements}, Ann. Inst. Fourier \textbf{61} (2011), 53-77.



\bibitem{LL} \textsc{E. H. Lieb, M. Loss}, Analysis. Second edition. Graduate Studies in Mathematics, $14$. American Mathematical Society, Providence, RI, 2001. ISBN:0-8218-2783-9.

\bibitem{MMS} \textsc{A. Maalaoui, L. Martinazzi, A. Schikorra}, \emph{Blow-up behaviour of a fractional Adams-Moser-Trudinger type inequality in odd dimension}, arXiv:1504.00254.

\bibitem{MM} \textsc{A. Malchiodi, L. Martinazzi}, \emph{Critical points of the Moser-Trudinger functional on a disk}, J. Eur. Math. Soc. (JEMS), \textbf{16} (2014), no. 5, 893Ð908.

\bibitem{PLL}\textsc{P.L. Lions}, \emph{The concentration-compactness principle in the calculus of variations. The limit case. I}. Rev. Mat. Iberoamericana \textbf{1} (1985), no. 1, 145-201.

\bibitem{poh} \textsc{S. I. Pohozaev}, \emph{The Sobolev embedding in the case
 $pl=n$}, Proc. Tech. Sci. Conf. on Adv. Sci. Research 1964-1965, Mathematics Section, Moskov. Energet. Inst. Moscow (1965), 158-170.

\bibitem{mar1} \textsc{L. Martinazzi}, \emph{A threshold phenomenon for embeddings of $H^m_0$ into Orlicz spaces}, Calc. Var. Partial Differential Equations \textbf{36} (2009), 493-506.

\bibitem{mar2} \textsc{L. Martinazzi}, \emph{Fractional Adams-Moser-Trudinger type inequalities}, Nonlinear Analysis \textbf{127} (2015) 263-278.

\bibitem{mos} \textsc{J. Moser,} \emph{A sharp form of an inequality by N. Trudinger}, Indiana Univ. Math. J. \textbf{20} (1970/71), 1077-1092.

\bibitem{Park} \textsc{Y. J. Park}, \emph{Fractional P\'olya-Szeg\H{o} inequality}. J. Chungcheong Math. Soc. \textbf{24} (2011), no. 2, 267-271.

\bibitem{RS} \textsc{F. Robert, M. Struwe}, \emph{Asymptotic profile for a fourth order PDE with critical exponential growth in dimension four}, Adv. Nonlin. Stud. \textbf{4} (2004), 397-415.

\bibitem{XR} \textsc{X. Ros-Oton, J. Serra} \emph{The Dirichlet problem for the fractional Laplacian: regularity up to the boundary}, J. Math. Pures Appl. \textbf{101} (2014), 275-302.

\bibitem{Ruf} \textsc{B. Ruf,} \emph{A sharp Trudinger-Moser type inequality for unbounded domains in $\R{2}$}, J. Funct. Analysis  \textbf{219} (2004), 340-367.

\bibitem{SV} \textsc{R. Servadei, E. Valdinoci}, \emph{Variational methods for non-local operators of elliptic type}. Discrete Continu. Dyn. Syst. \textbf{33} (2013), no. 5, 2105-2137.

\bibitem{Sil} \textsc{L. Silvestre}, \emph{Regularity of the obstacle problem for a fractional power of the Laplace operator}. Comm. Pure Appl. Math. \textbf{60} (2007), no. 1, 67-112.

\bibitem{str} \textsc{M. Struwe,} \emph{Critical points of embeddings of $H^{1,n}_0$ into Orlicz spaces}, Ann. Inst. H. Poincar\'e Anal. Non Lin\'eaire \textbf{5} (1988), 425-464.

\bibitem{tru} \textsc{N. S. Trudinger,} \emph{On embedding into Orlicz spaces and some applications}, J. Math. Mech. \textbf{17} (1967), 473-483.

\bibitem{XZ} \textsc{J. Xiao, Z. Zhai,} \emph{Fractional Sobolev, Moser-Trudinger, Morrey-Sobolev inequalities under Lorentz norms}, J. Math. Sci. \textbf{116} (2010), 357-376. 

\bibitem{yud} \textsc{V. I. Yudovich}, \emph{On certain estimates connected with integral operators and solutions of elliptic equations}, Dokl. Akad. Nank. SSSR \textbf{138} (1961), 805-808.

\end{thebibliography}
\end{document}